\documentclass[graybox,envcountsame,envcountsect]{svmult}

\smartqed 

\usepackage{mathptmx}       
\usepackage{helvet}         
\usepackage{courier}        

\usepackage{graphicx}        

\usepackage[bottom]{footmisc} 

\usepackage{amsfonts, amsmath, amssymb, mathrsfs, eucal}

\newcommand{\esssup}{\mathop{\mathrm{ess\,sup}}}
\newcommand{\essinf}{\mathop{\mathrm{ess\,inf}}}
\newcommand{\card}{\mathop{\mathrm{card}}}

\begin{document}

\title*{Ergodicity, Decisions, and Partial Information}
\author{Ramon van Handel}
\institute{Ramon van Handel \at Sherrerd Hall 227,
Princeton University, Princeton, NJ 08544, USA.
\email{rvan@princeton.edu}}
\maketitle

\vskip-3cm

\abstract{
In the simplest sequential decision problem for an ergodic stochastic
process $X$, at each time $n$ a decision $u_n$ is made as a function of
past observations $X_0,\ldots,X_{n-1}$, and a loss $l(u_n,X_n)$ is
incurred.  In this setting, it is known that one may choose (under a
mild integrability assumption) a decision strategy whose pathwise
time-average loss is asymptotically smaller than that of any other
strategy.  The corresponding problem in the case of partial information
proves to be much more delicate, however: if the process $X$ is not
observable, but decisions must be based on the observation of a
different process $Y$, the existence of pathwise optimal strategies  
is not guaranteed.
The aim of this paper is to exhibit connections between pathwise optimal 
strategies and notions from ergodic theory.  The sequential decision 
problem is developed in the general setting of an ergodic dynamical system 
$(\Omega,\mathcal{B},\mathbf{P},T)$ with partial information 
$\mathcal{Y}\subseteq\mathcal{B}$.  The existence of pathwise optimal 
strategies grounded in two basic properties: the conditional ergodic 
theory of the dynamical system, and the complexity of the loss function. 
When the loss function is not too complex, a general sufficient condition 
for the existence of pathwise optimal strategies is that the dynamical 
system is a conditional $K$-automorphism relative to the past observations 
$\bigvee_{n\ge 0} T^n \mathcal{Y}$.  If the conditional ergodicity 
assumption is strengthened, the complexity assumption can be weakened. 
Several examples demonstrate the interplay between complexity and 
ergodicity, which does not arise in the case of full information. Our 
results also yield a decision-theoretic characterization of weak mixing in 
ergodic theory, and establish pathwise optimality of ergodic nonlinear 
filters.}

\section{Introduction}

Let $X=(X_k)_{k\in\mathbb{Z}}$ be a stationary and ergodic stochastic 
process.  A decision maker must select at the beginning of each day $k$ 
a decision $u_k$ depending on the past observations 
$X_0,\ldots,X_{k-1}$.  At the end of the day, a loss $l(u_k,X_k)$ is 
incurred.  The decision maker would like to minimize her time-average loss
$$
	L_T(\mathbf{u}) = \frac{1}{T}\sum_{k=1}^T l(u_k,X_k).
$$
How should she go about selecting a decision strategy $\mathbf{u}= 
(u_k)_{k\ge 1}$?  

There is a rather trivial answer to this question.  Taking the 
expectation of the time-average loss, we obtain for any strategy
$\mathbf{u}$ using the tower property
\begin{align*}
	\mathbf{E}[L_T(\mathbf{u})] &=
	\mathbf{E}\Bigg[\frac{1}{T}\sum_{k=1}^T 
	\mathbf{E}[l(u_k,X_k)|X_0,\ldots,X_{k-1}]
	\Bigg] \\ &\ge
	\mathbf{E}\Bigg[\frac{1}{T}\sum_{k=1}^T 
	\min_{u}
	\mathbf{E}[l(u,X_k)|X_0,\ldots,X_{k-1}]
	\Bigg] =
	\mathbf{E}[L_T(\mathbf{\tilde u})],
\end{align*}
where $\mathbf{\tilde u}=(\tilde u_k)_{k\ge 1}$ is defined as
$\tilde u_k = \mathop{\mathrm{arg\,min}}_u 
\mathbf{E}[l(u,X_k)|X_0,\ldots,X_{k-1}]$ (we disregard for the moment
integrability and measurability issues, existence of minima, and the like;
such issues will be properly addressed in our results).  Therefore,
the strategy $\mathbf{\tilde u}$ minimizes the \emph{mean} time-average loss
$\mathbf{E}[L_T(\mathbf{u})]$.

However, there are conceptual reasons to be dissatisfied with this 
obvious solution.  In many decision problems, one only observes a single 
sample path of the process $X$.  For example, if $X_k$ is the return of 
a financial market in day $k$ and $L_T(\mathbf{u})$ is the loss of an 
investment strategy $\mathbf{u}$, only one sample path of the model is 
ever realized: we do not have the luxury of averaging our investment 
loss over multiple ``alternative histories''.  The choice of a strategy 
for which the mean loss is small does not guarantee, \emph{a priori}, 
that it will perform well on the one and only realization that happens 
to be chosen by nature. Similarly, if $X_k$ models the state of the 
atmosphere and $L_T(\mathbf{u})$ is the error of a weather prediction 
strategy, we face a similar conundrum.  In such situations, the use of 
stochastic models could be justified by some sort of ergodic theorem, 
which states that the mean behavior of the model with respect to 
different realizations captures its time-average behavior over a single 
sample path.  Such an ergodic theorem for sequential decisions was
obtained by Algoet \cite[Theorem 2]{Alg94} under a mild integrability assumption.

\begin{theorem}[Algoet \cite{Alg94}]
\label{thm:algoet}
Suppose that $|l(u,x)|\le\Lambda(x)$ with $\Lambda\in L\log L$.
Then
$$
	\liminf_{T\to\infty}\{L_T(\mathbf{u})-L_T(\mathbf{\tilde u})\}
	\ge 0\quad\mbox{a.s.}
$$
for every strategy $\mathbf{u}$: that is, the mean-optimal strategy
$\mathbf{\tilde u}$ is pathwise optimal.
\end{theorem}

The proof of this result follows from a simple martingale argument. What 
is remarkable is that the details of the model do not enter the picture 
at all: nothing is assumed on the properties of $X$ or $l$ beyond some 
integrability (ergodicity is not needed, and a similar result holds even 
in the absence of stationarity, cf.\ \cite[Theorem 3]{Alg94}).  This 
provides a universal justification for optimizing the mean loss: the 
much stronger pathwise optimality property is obtained ``for free.'' 


In the proof of Theorem \ref{thm:algoet}, it is essential that the 
decision maker has full information on the history $X_0,\ldots,X_{k-1}$ 
of the process $X$.  However, the derivation of the mean-optimal 
strategy can be done in precisely the same manner in the more general 
setting where only partial or noisy information is available.  To 
formalize this idea, let $Y=(Y_k)_{k\in\mathbb{Z}}$ be the stochastic 
process observable by the decision maker, and suppose that the pair 
$(X,Y)$ is stationary and ergodic.  The loss incurred at time $k$ is 
still $l(u_k,X_k)$, but now $u_k$ may depend on the observed data 
$Y_0,\ldots,Y_{k-1}$ only.  It is easily seen that in this setting, the 
mean-optimal strategy $\mathbf{\tilde u}$ is given by $\tilde u_k = 
\mathop{\mathrm{arg\,min}}_u\mathbf{E}[l(u,X_k)|Y_0,\ldots, Y_{k-1}]$, 
and it is tempting to assume that $\mathbf{\tilde u}$ is also 
pathwise optimal.  Surprisingly, this is very far from being the case.

\begin{example}[Weissman and Merhav \cite{WM04}]
\label{ex:merhav}
Let $X_0\sim\mathrm{Bernoulli}(1/2)$ and let $X_k=1-X_{k-1}$ and $Y_k=0$
for all $k$.  Then $(X,Y)$ is stationary and ergodic:
$Y_k=0$ indicates that we are in the setting of \emph{no}
information (that is, we must make blind decisions).
Consider the loss $l(u,x) = (u-x)^2$.  Then the mean-optimal strategy
$\tilde u_k=1/2$ satisfies $L_T(\mathbf{\tilde u})=1/4$ for all $T$.
However, the strategy $u_k=k\mathop{\mathrm{mod}}2$ satisfies
$L_T(\mathbf{u})=0$ for all $T$ with probability $1/2$.  Therefore, 
$\mathbf{\tilde u}$ is not pathwise optimal.  In fact, it is easily seen
that no pathwise optimal strategy exists.
\end{example}

Example \ref{ex:merhav} illustrates precisely the type of conundrum that 
was so fortuitously ruled out in the full information setting by Theorem 
\ref{thm:algoet}.  Indeed, it would be hard to argue that either 
$\mathbf{u}$ or $\mathbf{\tilde u}$ in Example \ref{ex:merhav} is 
superior: a gambler placing blind bets $u_k$ on a sequence of games with 
loss $l(u_k,X_k)$ may prefer either strategy depending on his demeanor. 
The example may seem somewhat artificial, however, as the hidden 
process $X$ has infinitely long memory; the gambler can therefore beat 
the mean-optimal strategy by simply guessing the outcome of the first 
game.  But precisely the same phenomenon can appear when $(X,Y)$ is
nearly memoryless.

\begin{example}
\label{ex:filt}
Let $(\xi_k)_{k\in\mathbb{Z}}$ be i.i.d.\ $\mathrm{Bernoulli}(1/2)$, and 
let $X_k=(\xi_{k-1},\xi_k)$ and $Y_k=|\xi_k-\xi_{k-1}|$ for all $k$. 
Then $(X,Y)$ is a stationary $1$-dependent sequence: $(X_k,Y_k)_{k\le 
n}$ and $(X_k,Y_k)_{k\ge n+2}$ are independent for every $k$.  We 
consider the loss $l(u,x) = (u-x_1)^2$.  It is easily seen 
that $X_k$ is independent of $Y_1,\ldots,Y_{k-1}$, so that the 
mean-optimal strategy $\tilde u_k=1/2$ satisfies $L_T(\mathbf{\tilde 
u})=1/4$ for all $T$.  On the other hand, note that
$\xi_{k-1} = (\xi_0+Y_1+\cdots+Y_{k-1})\mathop{\mathrm{mod}}2$.  It follows
that the strategy $u_k=(Y_1+\cdots+Y_{k-1})\mathop{\mathrm{mod}}2$
satisfies $L_T(\mathbf{u})=0$ for all $T$ with probability $1/2$.
\end{example}

Evidently, pathwise optimality cannot be taken for granted 
in the partial information setting even in the simplest of examples: in 
contrast to the full information setting, the existence of pathwise 
optimal strategies depends both on specific ergodicity properties of the 
model $(X,Y)$ and (as will be seen later) on the complexity on the loss 
$l$.  What mechanism is responsible for pathwise optimality under 
partial information is not very well understood. Weissman and Merhav 
\cite{WM04}, who initiated the study of this problem,
give a strong sufficient condition in the binary 
setting.  Little is known beyond their result, beside one particularly
special case of quadratic loss and additive noise considered by Nobel 
\cite{Nob03}.\footnote{ 
	It should be noted that the papers \cite{Alg94,WM04,Nob03}, in addition 
	to studying the pathwise optimality problem, also aim to obtain 
	\emph{universal} decision schemes that achieve the optimal asymptotic 
	loss without any knowledge of the law of $X$ (note that to compute 
	the mean-optimal strategy $\mathbf{\tilde u}$ one must know the joint 
	law of $(X,Y)$).  Such strategies ``learn'' the law of 
	$X$ on the fly from the observed data.  In the setting of partial 
	information, such universal schemes cannot exist 
	without very specific assumptions on the information structure: for 
	example, in the blind setting (cf.\ Example \ref{ex:merhav}), there 
	is no information and thus universal strategies cannot exist.
	What conditions are required for the existence of universal strategies
	is an interesting question that is beyond the scope of this paper.
}

The aim of this paper is twofold.  On the one hand, we will give general 
conditions for pathwise optimality under partial information, and 
explore some tradeoffs inherent in this setting.  On the other hand, we 
aim to exhibit some connections between the pathwise optimality problem 
and certain notions and problems in ergodic theory, such as conditional 
mixing and individual ergodic theorems for subsequences.  To make such 
connections in their most natural setting, we begin by rephrasing the 
decision problem in the general setting of ergodic dynamical systems.

\subsection{The dynamical system setting}

Let $T$ be an invertible measure-preserving transformation of a 
probability space $(\Omega,\mathcal{B},\mathbf{P})$.  $T$ defines the 
time evolution of the dynamical system 
$(\Omega,\mathcal{B},\mathbf{P},T)$: if the system is initially in state 
$\omega\in\Omega$, then at time $k$ the system is in the state 
$T^k\omega$.  The state of the system is not directly observable, 
however.  To model the available information, we fix a $\sigma$-field 
$\mathcal{Y}\subseteq\mathcal{B}$ of events that can be observed at a 
single time.  Therefore, if we have observed the system in the time 
interval $[m,n]$, the information contained in the observations is given 
by the $\sigma$-field 
$\mathcal{Y}_{m,n}=\bigvee_{k\in[m,n]}T^{-k}\mathcal{Y}$.

In this general setting, the decision problem is defined as follows. Let 
$\ell:U\times\Omega\to\mathbb{R}$ be a given loss function, where $U$ is the 
set of possible decisions.  At each time $k$, a decision $u_k$ is made
and a loss $\ell_k(u_k):= \ell(u_k,T^k\omega)$ is incurred.  The 
decision can only depend on the observations: that is, a 
strategy $\mathbf{u}=(u_k)_{k\ge 1}$ is admissible if $u_k$ is 
$\mathcal{Y}_{0,k}$-measurable for every $k$. The time-average loss is 
given by 
$$
	L_T(\mathbf{u}):=\frac{1}{T}\sum_{k=1}^T\ell_k(u_k).
$$
The basic question we aim to answer is whether there exists a
pathwise optimal strategy, that is, a strategy $\mathbf{u}^\star$ such that
for every admissible strategy $\mathbf{u}$
$$
	\liminf_{T\to\infty}\{L_T(\mathbf{u})-L_T(\mathbf{u}^\star)\}\ge 0
	\quad\mbox{a.s.}
$$
The stochastic process setting discussed above can be recovered as
a special case.

\begin{example} 
\label{ex:stproc}
Let $(X,Y)$ be a stationary and ergodic stochastic process, where $X_k$ 
takes values in the measurable space $(E,\mathcal{E})$ and $Y_k$ takes 
values in the measurable space $(F,\mathcal{F})$.  We can realize 
$(X,Y)$ as the coordinate process on the canonical path space 
$(\Omega,\mathcal{B},\mathbf{P})$ where $\Omega=E^\mathbb{Z}\times 
F^\mathbb{Z}$, $\mathcal{B}=\mathcal{E}^\mathbb{Z} 
\otimes\mathcal{F}^\mathbb{Z}$, and $\mathbf{P}$ is the law of $(X,Y)$. 
Let $T:\Omega\to\Omega$ be the canonical shift $(T(x,y))_n = 
(x_{n+1},y_{n+1})$.  Then $(\Omega,\mathcal{B},\mathbf{P},T)$ is an 
ergodic dynamical system.  If we choose the observation $\sigma$-field
$\mathcal{Y} = \sigma\{Y_0\}$ and the loss $\ell(u,\omega) =
l(u,X_1(\omega))$, we recover the decision problem with partial information
for the stochastic process $(X,Y)$ as it was introduced above.  More
generally, we could let the loss depend arbitrarily on future or past
values of $(X,Y)$.
\end{example}

Let us briefly discuss the connection between pathwise optimal strategies
and classical ergodic theorems.  The key observation in the derivation of
the mean-optimal strategy $\tilde u_k = \mathop{\mathrm{arg\,min}}_u
\mathbf{E}[\ell_k(u)|\mathcal{Y}_{0,k}]$ is that by the tower property
$$
	\mathbf{E}\Bigg[
	\frac{1}{T}\sum_{k=1}^T\ell_k(u_k)
	\Bigg]=
	\mathbf{E}\Bigg[
	\frac{1}{T}\sum_{k=1}^T
	\mathbf{E}[\ell_k(u_k)|\mathcal{Y}_{0,k}]
	\Bigg].
$$
As the summands on the right-hand side depend only on the observed
information, we can minimize inside the sum to obtain the mean-optimal
strategy $\mathbf{\tilde u}$.  Precisely the same considerations would 
show that $\mathbf{\tilde u}$ is pathwise optimal if we could prove the
ergodic counterpart of the tower property of conditional expectations
$$
	\frac{1}{T}\sum_{k=1}^T
	\{\ell_k(u_k)-\mathbf{E}[\ell_k(u_k)|\mathcal{Y}_{0,k}]\}
	\xrightarrow{T\to\infty}0\quad\mbox{a.s.}\quad\mbox{?}
$$
The validity of such a statement is far from obvious, however.

In the special case of blind decisions (that is, $\mathcal{Y}$ is the trivial
$\sigma$-field) the ``ergodic tower property'' reduces to the question of
whether, given $f_k(\omega):=\ell(u_k,\omega)$,
$$
	\frac{1}{T}\sum_{k=1}^T \{f_k-\mathbf{E}[f_k]\}\circ T^k
	\xrightarrow{T\to\infty}0\quad\mbox{a.s.}\quad\mbox{?}
$$
If the functions $f_k$ do not depend on $k$, this is precisely the 
individual ergodic theorem. However, an individual ergodic theorem need 
not hold for arbitrary sequences $f_k$.  Special cases of this problem 
have long been investigated in ergodic theory.  For example, if 
$f_k=a_kf$ for some fixed function $f$ and bounded sequence 
$(a_k)\subset\mathbb{R}$, the problem reduces to a weighted individual 
ergodic theorem, see \cite{BL85} and the references therein.  If 
$a_k\in\{0,1\}$ for all $k$, the problem reduces further to an 
individual ergodic theorem along a subsequence (at least if the sequence 
has positive density), cf.\ \cite{Con73,BL85} and the references 
therein.  A general characterization of such ergodic properties does not 
appear to exist, which suggests that it is probably very difficult to 
obtain necessary and sufficient conditions for pathwise optimality. The 
situation is better for mean (rather than individual) ergodic theorems, 
cf.\ \cite{BB86} and the references therein, and we will also obtain 
more complete results in a weaker setting.  

The more interesting case where the information $\mathcal{Y}$ is 
nontrivial provides additional complications. In this situation, the 
``ergodic tower property'' could be viewed as a type of 
\emph{conditional} ergodic theorem, in between the individual ergodic 
theorem and Algoet's result \cite{Alg94}. Our proofs are based on an 
elaboration of this idea.

\subsection{Some representative results}

The essence of our results is that, when the loss $\ell$ is not too 
complex, pathwise optimal strategies exist under suitable conditional 
mixing assumptions on the ergodic dynamical system 
$(\Omega,\mathcal{B},\mathbf{P},T)$. To this end, we introduce 
conditional variants of two standard notions in ergodic theory: weak 
mixing and $K$-automorphisms.

\begin{definition}
\label{def:cwkmix}
An invertible dynamical system $(\Omega,\mathcal{B},\mathbf{P},T)$ is
said to be \emph{conditionally weak mixing} relative to a $\sigma$-field
$\mathcal{Z}$ if for every $A,B\in\mathcal{B}$
$$
	\frac{1}{T}\sum_{k=1}^T
	|\mathbf{P}[A\cap T^kB|\mathcal{Z}]
	- \mathbf{P}[A|\mathcal{Z}]\,
	\mathbf{P}[T^kB|\mathcal{Z}]
	|
	\xrightarrow{T\to\infty}0\quad\mbox{in }L^1.
$$
\end{definition}

\begin{definition}
\label{def:ckauto}
An invertible dynamical system $(\Omega,\mathcal{B},\mathbf{P},T)$
is called a \emph{conditional $K$-automorphism} relative to a
$\sigma$-field $\mathcal{Z}\subset\mathcal{B}$ 
if there is a $\sigma$-field $\mathcal{X}\subset\mathcal{B}$ such that
\begin{enumerate}
\item $\mathcal{X}\subset T^{-1}\mathcal{X}$.
\item $\bigvee_{k=1}^\infty T^{-k}\mathcal{X}=\mathcal{B}$ $\mod\mathbf{P}$.
\item $\bigcap_{k=1}^\infty (\mathcal{Z}\vee T^k\mathcal{X})=
\mathcal{Z}$ $\mod\mathbf{P}$.
\vskip.2cm
\end{enumerate}
\end{definition}

When the $\sigma$-field $\mathcal{Z}$ is trivial, these definitions 
reduce\footnote{
	To be precise, our definitions are time-reversed with 
	respect to the textbook definitions; however, $T$ is a 
	$K$-automorphism if and only if $T^{-1}$ is a $K$-automorphism
	\cite[p.\ 110]{Wal82}, and the corresponding statement for
	weak mixing is trivial.  Therefore, our definitions are equivalent
	to those in \cite{Wal82}.
} to the usual notions of weak mixing and $K$-automorphism, cf.\ 
\cite{Wal82}.  Similar conditional mixing conditions also appear in the
ergodic theory literature, see \cite{Rud04} and the references
therein.

An easily stated consequence of our main results, for example,
is the following.

\begin{theorem}
\label{thm:soptfinite}
Suppose that $(\Omega,\mathcal{B},\mathbf{P},T)$ is a
conditional $K$-automorphism relative to $\mathcal{Y}_{-\infty,0}$.
Then the mean-optimal strategy $\mathbf{\tilde u}$ 
is pathwise optimal for every loss function $\ell:U\times\Omega\to
\mathbb{R}$ such that $U$ is finite and $|\ell(u,\omega)|\le\Lambda(\omega)$
with $\Lambda\in L^1$.
\end{theorem}

This result gives a general sufficient condition for pathwise 
optimality when the decision space $U$ is finite.  In the stochastic
process setting (Example \ref{ex:stproc}), the conditional $K$-property
would follow from the validity of the $\sigma$-field identity
$$
	\bigcap_{k=1}^\infty(\mathcal{Y}_{-\infty,0}\vee
	\mathcal{X}_{-\infty,-k})=\mathcal{Y}_{-\infty,0}
	\quad\mathop{\mathrm{mod}}\mathbf{P},
$$
where $\mathcal{X}_{-\infty,k}=\sigma\{X_i:i\le k\}$ (choose 
$\mathcal{X}:=\mathcal{X}_{-\infty,0}\vee\mathcal{Y}_{-\infty,0}$ in 
Definition \ref{def:ckauto}).  In the Markovian setting, this is a 
familiar identity in filtering theory: it is precisely the necessary and 
sufficient condition for the optimal filter to be ergodic, see
section \ref{sec:filt} below.  Our results therefore lead to a 
new pathwise optimality property of nonlinear filters.  Conversely, 
results from filtering theory yield a broad class of (even 
non-Markovian) models for which the conditional $K$-property can be 
verified \cite{vH09,TvH12}. It is interesting to note that despite the 
apparent similarity between the conditions for filter ergodicity and 
pathwise optimality, there appears to be no direct connection between 
these phenomena, and their proofs are entirely distinct.  Let us also 
note that, in the full information setting ($Y_k=X_k$) the conditional 
$K$-property holds trivially, which explains the deceptive simplicity of 
Algoet's result.

While the conditional ergodicity assumption of Theorem 
\ref{thm:soptfinite} is quite general, the requirement that the decision 
space $U$ is finite is a severe restriction on the complexity of the 
loss function $\ell$.  We have stated Theorem \ref{thm:soptfinite} here 
in order to highlight the basic ingredients for the existence of a 
pathwise optimal strategy. The assumption that $U$ is finite will be 
replaced by various complexity assumptions on the loss $\ell$; such 
extensions will be developed in the sequel.  While some complexity 
assumption on the loss is needed in the partial information setting,
there is a tradeoff between the complexity and ergodicity: if the notion 
of conditional ergodicity is strengthened, then the complexity assumption
on the loss can be weakened.

All our pathwise optimality results are corollaries of a general master 
theorem, Theorem \ref{thm:sopt} below, that ensures the existence of a 
pathwise optimal strategy under a certain uniform version of the 
$K$-automorphism property.  However, in the absence of further 
assumptions, this theorem does not ensure that the mean-optimal strategy 
$\mathbf{\tilde u}$ is in fact pathwise optimal: the pathwise optimal 
strategy constructed in the proof may be difficult to compute.  We do 
not know, in general, whether it is possible that a pathwise optimal 
strategy exists, while the mean-optimal strategy fails to be pathwise 
optimal.  In order to gain further insight into such questions, we 
introduce another notion of optimality that is intermediate between 
pathwise and mean optimality. A strategy $\mathbf{u}^\star$ is said to 
be weakly pathwise optimal if
$$
	\mathbf{P}[L_T(\mathbf{u})-L_T(\mathbf{u}^\star)\ge -\varepsilon]
	\xrightarrow{T\to\infty}1\quad\mbox{for every }\varepsilon>0.
$$
It is not difficult to show that if a weakly pathwise optimal strategy
exists, then the mean-optimal strategy $\mathbf{\tilde u}$ must also be
weakly pathwise optimal.  However, the notion of weak pathwise optimality
is distinctly weaker than pathwise optimality.  For example, we will prove
the following counterpart to Theorem \ref{thm:soptfinite}.

\begin{theorem}
\label{thm:woptfinite}
Suppose that $(\Omega,\mathcal{B},\mathbf{P},T)$ is
conditionally weak mixing relative to $\mathcal{Y}_{-\infty,0}$.
Then the mean-optimal strategy $\mathbf{\tilde u}$ 
is weakly pathwise optimal for every loss function $\ell:U\times\Omega\to
\mathbb{R}$ such that $U$ is finite and $|\ell(u,\omega)|\le\Lambda(\omega)$
with $\Lambda\in L^1$.
\end{theorem}

There is a genuine gap between Theorems \ref{thm:woptfinite} and 
\ref{thm:soptfinite}: in fact, a result of Conze \cite{Con73} on 
individual ergodic theorems for subsequences shows that there is 
a loss function $\ell$ such that for a generic (in the weak topology) 
weak mixing system, a mean-optimal blind strategy $\mathbf{\tilde u}$ 
fails to be pathwise optimal.

While weak pathwise optimality may not be as conceptually appealing as 
pathwise optimality, the weak pathwise optimality property is easier to 
characterize.  In particular, we will show that the conditional weak 
mixing assumption in Theorem \ref{thm:woptfinite} is not only 
sufficient, but also necessary, in the special case that $\mathcal{Y}$ 
is an invariant $\sigma$-field (that is, 
$\mathcal{Y}=T^{-1}\mathcal{Y}$). Invariance of $\mathcal{Y}$ is 
somewhat unnatural in decision problems, as it implies that no 
additional information is gained over time as more observations are 
accumulated.  On the other hand, invariance of $\mathcal{Z}$ in 
Definitions \ref{def:cwkmix} and \ref{def:ckauto} is precisely the 
situation of interest in applications of conditional mixing in ergodic 
theory (e.g., \cite{Rud04}).  The interest of this result is therefore 
that it provides a decision-theoretic characterization of the 
(conditional) weak mixing property.

\subsection{Organization of this paper}

The remainder of the paper is organized as follows.  In section 
\ref{sec:main}, we state and discuss the main results of this paper.  We 
also give a number of examples that illustrate various aspects of our 
results.  Our main results require two types of assumptions: conditional 
mixing assumptions on the dynamical system, and complexity assumptions 
on the loss.  In section \ref{sec:verifying} we discuss various methods 
to verify these assumptions, as well as further examples and consequences (such 
as pathwise optimality of nonlinear filters).  Finally, the proofs of 
our main results are given in section \ref{sec:proofs}.

\section{Main results}
\label{sec:main}

\subsection{Basic setup and notation} 

Throughout this paper, we will consider the following setting:
\begin{itemize} 
\item $(\Omega,\mathcal{B},\mathbf{P})$ is a probability space.
\item $\mathcal{Y}\subseteq\mathcal{B}$ is a sub-$\sigma$-field.
\item $T:\Omega\to\Omega$ is an invertible measure-preserving
ergodic transformation.
\item $(U,\mathcal{U})$ is a measurable space.
\end{itemize}
As explained in the introduction, we aim to make sequential decisions
in the ergodic dynamical system $(\Omega,\mathcal{B},\mathbf{P},T)$.
The decisions
take values in the decision space $U$, and the $\sigma$-field $\mathcal{Y}$
represents the observable part of the system.  We define
$$
	\mathcal{Y}_{m,n} = \bigvee_{k=m}^nT^{-k}\mathcal{Y}
	\qquad\mbox{for }{-\infty}\le m\le n\le\infty,
$$
that is, $\mathcal{Y}_{m,n}$ is the $\sigma$-field generated by the observations
in the time interval $[m,n]$.  An admissible decision strategy
must depend causally on the observed data.

\begin{definition}
A strategy $\mathbf{u}=(u_k)_{k\ge 1}$ is called \emph{admissible} 
if it is $\mathcal{Y}_{0,k}$-adapted, that is,
$u_k:\Omega\to U$ is $\mathcal{Y}_{0,k}$-measurable for every $k\ge 1$.
\end{definition}

It will be convenient to introduce the following notation.
For every $m\le n$, define
$$
	\mathbb{U}_{m,n} = \{u:\Omega\to U~{:}~u\mbox{ is }
	\mathcal{Y}_{m,n}\mbox{-measurable}\},\qquad
	\mathbb{U}_n = \bigcup_{-\infty<m\le n}\mathbb{U}_{m,n}.
$$
Thus a strategy $\mathbf{u}$ is admissible whenever $u_k\in\mathbb{U}_{0,k}$
for all $k$.  Note that $\mathbb{U}_n \subsetneq \mathbb{U}_{-\infty,n}$: 
this distinction will be essential for the validity of our results.

To describe the loss of a decision strategy, we introduce a loss function
$\ell$.
\begin{itemize}
\item $\ell:U\times\Omega\to\mathbb{R}$ is a measurable function and
$|\ell(u,\omega)|\le\Lambda(\omega)$ with $\Lambda\in
L^1$.
\end{itemize}
If $|\ell(u,\omega)|\le\Lambda(\omega)$ with $\Lambda\in
L^p$, the loss is said to be dominated in $L^p$.
As indicated above, we will always assume\footnote{
	Non-dominated loss functions may also be of significant interest,
	see \cite{Nob03} for example.  We will restrict attention to
	dominated loss functions, however, which suffice in many cases of 
	interest.
}
that our loss functions are
dominated in $L^1$.

The loss function $\ell(u,\omega)$ represents the cost incurred by 
the decision $u$ when the system is in state $\omega$.  In particular,
the cost of the decision $u_k$ at time $k$ is given by $\ell(u_k,T^k\omega)
=\ell_k(u_k)$, where we define for notational simplicity
$$
	\ell_n(u):\Omega\to\mathbb{R},\qquad
	\ell_n(u)(\omega)=\ell(u,T^n\omega).
$$
Our aim is to select an admissible strategy $\mathbf{u}$
that minimizes the time-average loss 
$$
	L_T(\mathbf{u}) = \frac{1}{T}\sum_{k=1}^T \ell_k(u_k)
$$
in a suitable sense.

\begin{definition}
\label{def:sopt}
An admissible strategy $\mathbf{u}^\star$ is \emph{pathwise 
optimal} if
$$
	\liminf_{T\to\infty}\{L_T(\mathbf{u})-L_T(\mathbf{u}^\star)\}
	\ge 0\quad\mbox{a.s.}
$$
for every admissible strategy $\mathbf{u}$.
\end{definition}

\begin{definition}
\label{def:wopt}
An admissible strategy $\mathbf{u}^\star$ is \emph{weakly pathwise 
optimal} if
$$
	\mathbf{P}[
	L_T(\mathbf{u})-L_T(\mathbf{u}^\star) \ge -\varepsilon
	]
	\xrightarrow{T\to\infty}1
	\quad\mbox{for every }\varepsilon>0
$$
for every admissible strategy $\mathbf{u}$.
\end{definition}

\begin{definition}
\label{def:mopt}
An admissible strategy $\mathbf{u}^\star$ is \emph{mean
optimal} if
$$
	\liminf_{T\to\infty}\{\mathbf{E}[L_T(\mathbf{u})]-
	\mathbf{E}[L_T(\mathbf{u}^\star)]\}
	\ge 0
$$
for every admissible strategy $\mathbf{u}$.
\end{definition}

These notions of optimality are progressively weaker: a pathwise optimal 
strategy is clearly weakly pathwise optimal, and a weakly pathwise 
optimal strategy is mean optimal (as the loss function is assumed to be
dominated in $L^1$).

In the introduction, it was stated that $\tilde u_k=
\mathop{\mathrm{arg\,min}}_{u\in U}\mathbf{E}[\ell_k(u)|\mathcal{Y}_{0,k}]$ 
defines a mean-optimal strategy.  This disregards some technical issues,
as the $\mathrm{arg\,min}$ may not exist or be measurable.
It suffices, however, to consider a slight reformulation.

\begin{lemma}
\label{lem:meanopt}
There exists an admissible strategy $\mathbf{\tilde u}$ such that
$$
	\mathbf{E}[\ell_k(\tilde u_k)|\mathcal{Y}_{0,k}]
	\le
	\essinf_{u\in\mathbb{U}_{0,k}}
	\mathbf{E}[\ell_k(u)|\mathcal{Y}_{0,k}] +
	k^{-1}\quad\mbox{a.s.}
$$
for every $k\ge 1$. In particular, $\mathbf{\tilde u}$ is mean-optimal.
\end{lemma}

\begin{proof}
It follows from the construction of the essential supremum
\cite[p.\ 49]{Pol02} that there exists a countable family
$(U^n)_{n\in\mathbb{N}}\subseteq\mathbb{U}_{0,k}$ such that
$$
	\essinf_{u\in\mathbb{U}_{0,k}}\mathbf{E}[\ell_k(u)|\mathcal{Y}_{0,k}] =
	\inf_{n\in\mathbb{N}}\mathbf{E}[\ell_k(U^n)|\mathcal{Y}_{0,k}].
$$
Define the random variable
$$
	\tau = \inf\Big\{n:
	\mathbf{E}[\ell_k(U^n)|\mathcal{Y}_{0,k}]\le
	\essinf_{u\in\mathbb{U}_{0,k}}\mathbf{E}[\ell_k(u)|\mathcal{Y}_{0,k}]
	+k^{-1}\Big\}.
$$
Note that $\tau<\infty$ a.s.\ as
$\essinf_{u\in\mathbb{U}_{0,k}}\mathbf{E}[\ell_k(u)|\mathcal{Y}_{0,k}]\ge
-\mathbf{E}[\Lambda\circ T^k|\mathcal{Y}_{0,k}]>-\infty$ a.s.
We therefore define $\tilde u_k=U^\tau$. To show that $\mathbf{\tilde u}$ 
is mean optimal, it suffices to note that
$$
	\mathbf{E}[L_T(\mathbf{u})] -
	\mathbf{E}[L_T(\mathbf{\tilde u})] =
	\frac{1}{T}\sum_{k=1}^T
	\mathbf{E}\Big[
	\mathbf{E}[\ell_k(u_k)|\mathcal{Y}_{0,k}]
	- \mathbf{E}[\ell_k(\tilde u_k)|\mathcal{Y}_{0,k}]
	\Big]
	\ge
	-\frac{1}{T}\sum_{k=1}^T k^{-1}
$$
for any admissible strategy $\mathbf{u}$ and $T\ge 1$.
\qed\end{proof}

In particular, we emphasize that a mean-optimal strategy $\mathbf{\tilde 
u}$ always exists.  In the remainder of this paper, we will fix a 
mean-optimal strategy $\mathbf{\tilde u}$ as in Lemma \ref{lem:meanopt}.

\subsection{Pathwise optimality}
\label{sec:popt}

Our results on the existence of pathwise optimal strategies are all 
consequences of one general result, Theorem \ref{thm:sopt}, that will be 
stated presently. The essential assumption of this general result is 
that the properties of the conditional $K$-automorphism (Definition 
\ref{def:ckauto}) hold uniformly with respect to the loss function 
$\ell$.  Note that, in principle, the assumptions of this result do not 
imply that $(\Omega,\mathcal{B},\mathbf{P},T)$ is a conditional 
$K$-automorphism, though this will frequently be the case.

\begin{theorem}[Pathwise optimality]
\label{thm:sopt}
Suppose that for some $\sigma$-field $\mathcal{X}\subset\mathcal{B}$
\begin{enumerate}
\item $\mathcal{X}\subset T^{-1}\mathcal{X}$.
\item The following martingales converge uniformly:
\begin{align*}
	&\esssup_{u\in\mathbb{U}_0}\big|
	\mathbf{E}[\ell_0(u)|\mathcal{Y}_{-\infty,0}\vee T^{-n}\mathcal{X}]-\ell_0(u)
	\big|\xrightarrow{n\to\infty}0 \quad\mbox{in }L^1, \\
	&\esssup_{u\in\mathbb{U}_0}\big|
	\mathbf{E}[\ell_0(u)|\mathcal{Y}_{-\infty,0}\vee T^{n}\mathcal{X}]-
	\mathbf{E}[\ell_0(u)|\textstyle{\bigcap_{k=1}^\infty}
	(\mathcal{Y}_{-\infty,0}\vee T^{k}\mathcal{X})]
	\big|\xrightarrow{n\to\infty}0\quad\mbox{in }L^1.
\end{align*}
\item The remote past does not affect the asymptotic loss:
$$
	L^\star :=
	\mathbf{E}\bigg[
	\essinf_{u\in\mathbb{U}_0}\mathbf{E}[\ell_0(u)|
	\mathcal{Y}_{-\infty,0}]
	\bigg] =
	\mathbf{E}\bigg[
	\essinf_{u\in\mathbb{U}_0}\mathbf{E}[\ell_0(u)|
	\textstyle{\bigcap_{k=1}^\infty}
        (\mathcal{Y}_{-\infty,0}\vee T^{k}\mathcal{X})]
	\bigg].
$$
\end{enumerate}
Then there exists an admissible strategy $\mathbf{u}^\star$ such 
that for every admissible strategy $\mathbf{u}$
$$
	\liminf_{T\to\infty} \{L_T(\mathbf{u})-L_T(\mathbf{u}^\star)\}
	\ge 0\quad\mbox{a.s.},
	\qquad
	\lim_{T\to\infty}L_T(\mathbf{u}^\star) = L^\star
	\quad
	\mbox{a.s.},
$$
that is, $\mathbf{u}^\star$ is pathwise optimal and
$L^\star$ is the optimal long time-average loss.
\end{theorem}

The proof of this result will be given in section \ref{sec:proofsopt} 
below.

Before going further, let us discuss the conceptual nature of the 
assumptions of Theorem
\ref{thm:sopt}.  The assumptions encode two separate requirements:
\begin{enumerate}
\item Assumption 3 of Theorem \ref{thm:sopt} should be viewed as a
mixing assumption on the dynamical system 
$(\Omega,\mathcal{B},\mathbf{P},T)$ that is tailored to the decision
problem.  Indeed, 
$\mathcal{Y}_{-\infty,0}$ represents the information contained in the 
observations, while $\bigcap_{k=1}^\infty (\mathcal{Y}_{-\infty,0} \vee 
T^k\mathcal{X})$ includes in addition the remote past of the generating 
$\sigma$-field $\mathcal{X}$.  The assumption states that 
knowledge of the remote past of the unobserved part of the model cannot 
be used to improve our present decisions.
\item Assumption 2 of Theorem \ref{thm:sopt} should be viewed
as a complexity assumption on the loss function $\ell$.  Indeed,
in the absence of the essential suprema, these statements
hold automatically by the martingale convergence theorem.  The assumption
requires that the convergence is in fact uniform in $u\in\mathbb{U}_0$.
This will be the case when the loss function is not too complex.  
\end{enumerate}
The assumptions of Theorem \ref{thm:sopt} can be 
verified in many cases of interest.  In section \ref{sec:verifying}
below, we will discuss various methods that can be used to verify 
both the conditional mixing and the complexity assumptions of 
Theorem \ref{thm:sopt}.

In general, neither the conditional mixing nor the complexity assumption can 
be dispensed with in the presence of partial information.

\begin{example}[Assumption 3 is essential]
\label{ex:asp3}
We have seen in Examples \ref{ex:merhav} and \ref{ex:filt} in the 
introduction that no pathwise optimal strategy exists.  In both these
examples Assumption 2 is satisfied, that is, the loss function is not 
too complex (this will follow from general complexity results, cf.\ 
Example \ref{ex:meansq} in
section \ref{sec:verifying} below).  On the other hand, it is easily 
seen that the conditional mixing Assumption 3 is violated.
\end{example}

\begin{example}[Assumption 2 is essential]
\label{ex:asp2}
Let $X=(X_k)_{k\in\mathbb{Z}}$ be the stationary Markov chain in 
$[0,1]$ defined by $X_{k+1}=(X_k+\varepsilon_{k+1})/2$ for all $k$,
where $(\varepsilon_k)_{k\in\mathbb{Z}}$ is an i.i.d.\ sequence of
$\mathrm{Bernoulli}(1/2)$ random variables.  
We consider the setting of blind decisions with the loss function 
$\ell_k(u)=\lfloor 2^uX_k\rfloor\mathop{\mathrm{mod}}2$, $u\in U=\mathbb{N}$.  Note that
$$
	X_k = \sum_{i=0}^\infty 2^{-i-1}\varepsilon_{k-i},\qquad
	\quad\ell_k(u) = \varepsilon_{k-u+1}.
$$
We claim that no pathwise optimal strategy can exist.  Indeed, consider
for fixed $r\ge 0$ the strategy $\mathbf{u}^r$ such that $u^r_k=k+r$.
Then $\ell_k(u_k^r) = \varepsilon_{1-r}$ for all $k$.  Therefore,
$$
	\varepsilon_{1-r} - \limsup_{T\to\infty}L_T(\mathbf{u}^\star) =
	\liminf_{T\to\infty}\{L_T(\mathbf{u}^r)-L_T(\mathbf{u}^\star)\}
	\ge 0\quad\mbox{a.s.}\quad\mbox{for all }r\ge 0
$$
for every pathwise optimal strategy $\mathbf{u}^\star$.  In particular,
$$
	0=
	\inf_{r\ge 0}\varepsilon_{1-r} \ge \limsup_{T\to\infty}L_T(\mathbf{u}^\star)
	\ge \liminf_{T\to\infty}L_T(\mathbf{u}^\star)\ge 0\quad\mbox{a.s.}
$$
As $|L_T(\mathbf{u}^\star)|\le 1$ for all $T$, it follows by dominated
convergence that a pathwise optimal strategy $\mathbf{u}^\star$ must satisfy
$\mathbf{E}[L_T(\mathbf{u}^\star)]\to 0$ as $T\to\infty$.
But clearly $\mathbf{E}[L_T(\mathbf{u})]=1/2$ for every $T$ and strategy
$\mathbf{u}$, which entails a contradiction.

Nonetheless, in this example the dynamical system is a $K$-automorphism 
(even a Bernoulli shift), so that that Assumption 3 is easily satisfied. 
As no pathwise optimal strategy exists, this must be caused by the 
failure of Assumption 2.  For example, for the natural choice 
$\mathcal{X}=\sigma\{X_k:k\le 0\}$, Assumption 3 holds as 
$\bigcap_kT^k\mathcal{X}$ is trivial by the Kolmogorov zero-one law, but 
it is easily seen that the second equation of Assumption 2 fails.
Note that the function $l(u,x)=\lfloor 2^ux\rfloor\mathop{\mathrm{mod}}2$ becomes
increasingly oscillatory as $u\to\infty$; this is precisely the type of
behavior that obstructs uniform convergence in Assumption 2 (akin to
``overfitting'' in statistics).
\end{example}

\begin{example}[Assumption 2 is essential, continued]
\label{ex:asp2ctd}
In the previous example, pathwise optimality fails due to failure of the 
second equation of Assumption 2.  We now give a variant of this example
where the first equation of Assumption 2 fails.

Let $X=(X_k)_{k\in\mathbb{Z}}$ be an i.i.d.\ sequence of
$\mathop{\mathrm{Bernoulli}}(1/2)$ random variables.
We consider the setting of blind decisions with the loss
function $\ell_k(u)=X_{k+u}$, $u\in U=\mathbb{N}$.  We claim that
no pathwise optimal strategy can exist.  Indeed, consider for
$r=0,1$ the strategy $\mathbf{u}^r$ defined by 
$u_k=2^{r+n+1}-k$ for $2^n\le k<2^{n+1}$, $n\ge 0$.  Then
$$
	L_{2^n-1}(\mathbf{u}^r) =
	\frac{1}{2^n-1}
	\sum_{m=0}^{n-1}\sum_{k=2^m}^{2^{m+1}-1} X_{k+u_k} =
	\frac{2^n}{2^n-1}
	\sum_{m=0}^{n-1} 2^{-(n-m)} X_{2^{r+m+1}}.
$$
Suppose that $\mathbf{u}^\star$ is pathwise optimal.  Then
$$
	\liminf_{T\to\infty}\mathbf{E}[L_T(\mathbf{u}^0)
	\wedge L_T(\mathbf{u}^1)-L_T(\mathbf{u^\star})] \ge
	\mathbf{E}\Big[
	\liminf_{T\to\infty}\{
	L_T(\mathbf{u}^0)
	\wedge L_T(\mathbf{u}^1)-L_T(\mathbf{u^\star})\}\Big]
	\ge 0.
$$
But a simple computation shows that
$\mathbf{E}[L_{2^n-1}(\mathbf{u}^0)
\wedge L_{2^n-1}(\mathbf{u}^1)]$ converges as $n\to\infty$ to a quantity
strictly less than $1/2=\mathbf{E}[L_T(\mathbf{u}^\star)]$, so that
we have a contradiction.

Nonetheless, in this example Assumption 3 and the second line of 
Assumption 2 are easily satisfied, e.g., for the natural choice
$\mathcal{X}=\sigma\{X_k:k\le 0\}$.  However, the first line of
Assumption 2 fails, and indeed no pathwise optimal strategy exists.
\end{example}

It is evident from the previous examples that an assumption on both 
conditional mixing and on complexity of the loss function is needed, in 
general, to ensure existence of a pathwise optimal strategy.  In this 
light, the complete absence of any such assumptions in the full 
information case is surprising. The explanation is simple, however: all 
assumptions of Theorem \ref{thm:sopt} are automatically satisfied in 
this case.

\begin{example}[Full information]
\label{ex:fullinf}
Let $X=(X_k)_{k\in\mathbb{Z}}$ be any stationary ergodic process, and
consider the case of full information: that is, we choose the observation
$\sigma$-field $\mathcal{Y}=\sigma\{X_0\}$ and the loss
$\ell(u,\omega)=l(u,X_1(\omega))$.  Then all assumptions of
Theorem \ref{thm:sopt} are satisfied: indeed, if we choose
$\mathcal{X}=\sigma\{X_k:k\le 0\}$, then 
$\mathcal{Y}_{-\infty,0}=\mathcal{Y}_{-\infty,0}\vee T^k\mathcal{X}$ for
all $k\ge 0$, so that Assumption 3 and the second line of Assumption 2
hold trivially.  Moreover, $\ell_0(u)$ is $T^{-k}\mathcal{X}$-measurable
for every $u\in\mathbb{U}_0$ and $k\ge 1$, and thus the first line
of Assumption 2 holds trivially.  It follows that in the full 
information setting, a pathwise optimal strategy always exists.
\end{example}

In a sense, Theorem \ref{thm:sopt} provides additional insight even in the 
full information setting: it provides an explanation as to why the case 
of full information is so much simpler than the partial information 
setting.  Moreover, Theorem \ref{thm:sopt} provides an explicit 
expression for the optimal asymptotic loss $L^\star$, which
is not given in \cite{Alg94}.\footnote{
	In \cite[Appendix II.B]{Alg94} it is shown that under a
	continuity assumption on the loss function $l$, the optimal asymptotic
	loss in the full information setting is given by
	$\mathbf{E}[\inf_u\mathbf{E}[l(u,X_1)|X_{0},X_{-1},\ldots]]$.
	However, a counterexample is given of a discontinuous loss function
	for which this expression does not yield the optimal asymptotic loss.
	The key difference with the expression for $L^\star$ given
	in Theorem \ref{thm:sopt} is that in the latter the essential
	infimum runs over $u\in\mathbb{U}_0$, while it is implicit in
	\cite{Alg94} that the infimum in the above expression is an 
	essential infimum over $u\in\mathbb{U}_{-\infty,0}$.
	As the counterexample in \cite{Alg94} shows, these quantities 
	need not coincide in the absence of continuity assumptions.}

However, it should be emphasized that Theorem \ref{thm:sopt} does not 
state that the mean-optimal strategy $\mathbf{\tilde u}$ is pathwise 
optimal; it only guarantees the existence of some pathwise optimal 
strategy $\mathbf{u}^\star$.  In contrast, in the full information 
setting, Theorem \ref{thm:algoet} ensures pathwise optimality of the 
mean-optimal strategy.  This is of practical importance, as the 
mean-optimal strategy can in many cases be computed explicitly or by 
efficient numerical methods, while the pathwise optimal strategy 
constructed in the proof of Theorem \ref{thm:sopt} may be difficult to 
compute.  We do not know whether it is possible in the general setting 
of Theorem \ref{thm:sopt} that a pathwise optimal strategy exists, but 
that the mean-optimal strategy $\mathbf{\tilde u}$ is not pathwise 
optimal.  Pathwise optimality of the mean-optimal strategy 
$\mathbf{\tilde u}$ can be shown, however, under somewhat stronger 
assumptions.  The following corollary is proved in section 
\ref{sec:proofsmeanopt} below.

\begin{corollary}
\label{cor:smeanopt}
Suppose that for some $\sigma$-field $\mathcal{X}\subset\mathcal{B}$
\begin{enumerate}
\item $\mathcal{X}\subset T^{-1}\mathcal{X}$.
\item The following martingales converge uniformly:
\begin{align*}
	\esssup_{u\in\mathbb{U}_0}&\big|
	\mathbf{E}[\ell_0(u)|\mathcal{Y}_{-\infty,0}\vee T^{-n}\mathcal{X}]-\ell_0(u)
	\big|\xrightarrow{n\to\infty}0\quad\mbox{in }L^1, \\
	\esssup_{u\in\mathbb{U}_0}&\big|
	\mathbf{E}[\ell_0(u)|\mathcal{Y}_{-\infty,0}\vee T^{n}\mathcal{X}]-
	\mathbf{E}[\ell_0(u)|\textstyle{\bigcap_{k=1}^\infty}
	(\mathcal{Y}_{-\infty,0}\vee T^{k}\mathcal{X})]
	\big|\xrightarrow{n\to\infty}0\quad\mbox{in }L^1, \\
	\esssup_{u\in\mathbb{U}_{-n,0}}&\big|
	\mathbf{E}[\ell_0(u)|\mathcal{Y}_{-n,0}]-
	\mathbf{E}[\ell_0(u)|\mathcal{Y}_{-\infty,0}]
	\big|\xrightarrow{n\to\infty}0 \quad\mbox{a.s.}
\end{align*}
\item The remote past does not affect the present:
$$
	\mathbf{E}[\ell_0(u)|\mathcal{Y}_{-\infty,0}]
	=
	\mathbf{E}[\ell_0(u)|
	\textstyle{\bigcap_{k=1}^\infty}
        (\mathcal{Y}_{-\infty,0}\vee T^{k}\mathcal{X})]
	\qquad\mbox{for all }
	u\in\mathbb{U}_0.
$$
\end{enumerate}
Then the mean-optimal strategy $\mathbf{\tilde u}$ (Lemma 
\ref{lem:meanopt}) satisfies $L_T(\mathbf{\tilde u})\to L^\star$ a.s.\ 
as $T\to\infty$.  In particular, it follows from Theorem \ref{thm:sopt} 
that $\mathbf{\tilde u}$ is pathwise optimal.
\end{corollary}

The assumptions of Corollary \ref{cor:smeanopt} are stronger than those 
of Theorem \ref{thm:sopt} in two respects.  First, Assumption 3 is 
slightly strengthened; however, this is a very mild requirement.  More 
importantly, a third martingale is assumed to converge uniformly 
(pathwise!) in Assumption 2.  The latter is not an innocuous 
requirement: while the assumption holds in many cases of interest, 
substantial regularity of the loss function is needed (see section 
\ref{sec:unicplx} for further discussion).  In particular, this 
requirement is not automatically satisfied in the case of full 
information, and Theorem \ref{thm:algoet} therefore does not follow in 
its entirety from our results.  It remains an open question whether it 
is possible to establish pathwise optimality of the mean-optimal strategy 
$\mathbf{\tilde u}$ under a substantial weakening of the assumptions
of Corollary \ref{cor:smeanopt}.

A particularly simple regularity assumption on the loss is that the 
decision space $U$ is finite.  In this case uniform convergence is 
immediate, so that the assumptions of Corollary \ref{cor:smeanopt} 
reduce essentially to the $\mathcal{Y}_{-\infty,0}$-conditional 
$K$-property. Therefore, evidently Corollary \ref{cor:smeanopt} implies 
Theorem \ref{thm:soptfinite}.  More general conditions that ensure the 
validity of the requisite assumptions will be discussed in section 
\ref{sec:verifying}.

\subsection{Weak pathwise optimality}

In the previous section, we have seen that a pathwise optimal strategy 
$\mathbf{u}^\star$ exists under general assumptions.  However, 
unlike in the full information case, it is not clear whether in general 
(without a nontrivial complexity assumption) the mean-optimal strategy 
$\mathbf{\tilde u}$ is pathwise optimal.  In the present section, we 
will aim to obtain some additional insight into this issue by 
considering the notion of weak pathwise optimality (Definition 
\ref{def:wopt}) that is intermediate between pathwise optimality and 
mean optimality.  This notion is more regularly behaved than pathwise 
optimality; in particular, it is straightforward to prove the following 
simple result.

\begin{lemma}
\label{lem:wkimplmn}
Suppose that a weakly pathwise optimal strategy $\mathbf{u}^\star$ exists.
Then the mean-optimal strategy $\mathbf{\tilde u}$ is also weakly pathwise
optimal.
\end{lemma}

\begin{proof}
Let $\Lambda_T=\frac{1}{T}\sum_{k=1}^T\Lambda\circ T^k$. 
As $|L_T(\mathbf{u})|\le\Lambda_T$ for any strategy $\mathbf{u}$,
we have
$$
	\mathbf{E}[(L_T(\mathbf{\tilde u})-L_T(\mathbf{u}^\star))_-] \le
	\varepsilon\, \mathbf{P}[L_T(\mathbf{\tilde u})-L_T(\mathbf{u}^\star)\ge-\varepsilon]
	+ \mathbf{E}[2\Lambda_T\,
	\mathbf{1}_{L_T(\mathbf{\tilde u})-L_T(\mathbf{u}^\star)<-\varepsilon}]
$$
for any $\varepsilon>0$.  Note that the sequence $(\Lambda_T)_{T\ge 
1}$ is uniformly integrable as $\Lambda_T\to\mathbf{E}[\Lambda]$ in 
$L^1$ by the ergodic theorem. Therefore, using weak pathwise optimality 
of $\mathbf{u}^\star$, it follows that $\mathbf{E}[(L_T(\mathbf{\tilde 
u})-L_T(\mathbf{u}^\star))_-] \to 0$ as $T\to\infty$.  We therefore
have
$$
	\limsup_{T\to\infty}
	\mathbf{E}[|L_T(\mathbf{\tilde u})-L_T(\mathbf{u}^\star)|] =
	-\liminf_{T\to\infty}
	\{\mathbf{E}[L_T(\mathbf{u}^\star)]-\mathbf{E}[L_T(\mathbf{\tilde u})]
	\} \le 0
$$
by mean-optimality of $\mathbf{\tilde u}$.  It follows easily that
$\mathbf{\tilde u}$ is also pathwise optimal.
\qed\end{proof}

While Theorem \ref{thm:sopt} does not ensure that the mean-optimal 
strategy $\mathbf{\tilde u}$ is pathwise optimal, the previous lemma 
guarantees that $\mathbf{\tilde u}$ is at least weakly pathwise optimal. 
However, we will presently show that the latter conclusion may follow 
under considerably weaker assumptions than those of of Theorem 
\ref{thm:sopt}.  Indeed, just as pathwise optimality was established for 
conditional $K$-automorphisms, we will establish weak optimality for 
conditionally weakly mixing automorphisms.

Let us begin by developing a general result on weak pathwise optimality,
Theorem \ref{thm:wkopt} below, that plays the role of Theorem \ref{thm:sopt}
in the present setting.  The essential assumption of this general result
is that the conditional weak mixing property (Definition \ref{def:cwkmix})
holds uniformly with respect to the loss function $\ell$.  For
simplicity of notation, let us define as in Theorem \ref{thm:sopt}
the optimal asymptotic loss
$$
	L^\star :=
	\mathbf{E}\Big[
	\essinf_{u\in \mathbb{U}_0}
	\mathbf{E}[\ell_0(u)|\mathcal{Y}_{-\infty,0}]
	\Big]
$$
(let us emphasize, however, the Assumption 3 of Theorem \ref{thm:sopt} 
need not hold in the present setting!)  In addition, let us define
the modified loss functions
$$
	\bar\ell_0(u) := 
	\ell_0(u)-\mathbf{E}[\ell_0(u)|\mathcal{Y}_{-\infty,0}],\qquad
	\bar\ell_0^M(u) := 
	\ell_0(u)\mathbf{1}_{\Lambda\le M}
	-\mathbf{E}[\ell_0(u)\mathbf{1}_{\Lambda\le M}
	|\mathcal{Y}_{-\infty,0}].
$$
The proof of the following theorem will be
given in section \ref{sec:proofwkopt}.

\begin{theorem}
\label{thm:wkopt}
Suppose that the uniform conditional mixing assumption
$$
	\lim_{M\to\infty}
	\limsup_{T\to\infty}
	\bigg\|
	\frac{1}{T}\sum_{k=1}^T
	\esssup_{u,u'\in\mathbb{U}_{0}}
	|\mathbf{E}[\{\bar\ell_0^M(u)\circ T^{-k}\} ~\bar\ell_0^M(u')|
	\mathcal{Y}_{-\infty,0}]|
	\bigg\|_1 = 0
$$
holds. Then the mean-optimal strategy $\mathbf{\tilde u}$
is weakly pathwise optimal, and the optimal long time-average loss 
satisfies the ergodic theorem $L_T(\mathbf{\tilde u})\to L^\star$ in $L^1$.
\end{theorem}

\begin{remark}
We have assumed throughout that the loss function $\ell$ is dominated in 
$L^1$.  If the loss is in fact dominated in $L^2$, that is,
$|\ell(u,\omega)|\le\Lambda(\omega)$ with $\Lambda\in L^2$, then the
assumption of Theorem \ref{thm:wkopt} is evidently implied by the
natural assumption
$$
	\frac{1}{T}\sum_{k=1}^T
	\esssup_{u,u'\in\mathbb{U}_{0}}
	|\mathbf{E}[\{\bar\ell_0(u)\circ T^{-k}\}~\bar\ell_0(u')|
	\mathcal{Y}_{-\infty,0}]|
	\xrightarrow{T\to\infty}0\quad\mbox{in }L^1, 
$$ 
and in this case $L_T(\mathbf{\tilde u})\to L^\star$ in $L^2$ (by 
dominated convergence). The additional truncation in Theorem 
\ref{thm:wkopt} is included only to obtain a result that holds in $L^1$. 
\end{remark}

Conceptually, as in Theorem \ref{thm:sopt}, the assumption of Theorem 
\ref{thm:wkopt} combines a conditional mixing assumption and a 
complexity assumption.  Indeed, the conditional weak mixing property 
relative to $\mathcal{Y}_{-\infty,0}$ (Definition \ref{def:cwkmix}) 
implies that
$$
	\frac{1}{T}\sum_{k=1}^T
	|\mathbf{E}[\{f\circ T^{-k}\}~g|\mathcal{Y}_{-\infty,0}]
	- \mathbf{E}[f\circ T^{-k}|\mathcal{Y}_{-\infty,0}]\,
	\mathbf{E}[g|\mathcal{Y}_{-\infty,0}]
	|
	\xrightarrow{T\to\infty}0
	\quad\mbox{in }L^1
$$
for every $f,g\in L^2$ (indeed, for simple functions $f,g$ this follows 
directly from the definition, and the claim for general $f,g$ follows by 
approximation in $L^2$).  Therefore, in the absence of the essential supremum,
the assumption of
Theorem \ref{thm:wkopt} reduces essentially to the assumption that the
dynamical system $(\Omega,\mathcal{B},\mathbf{P},T)$ is conditionally
weak mixing relative to $\mathcal{Y}_{-\infty,0}$.  
However, Theorem \ref{thm:wkopt} requires in addition that
the convergence in the definition of the conditional weak mixing property
holds uniformly with respect to the possible decisions $u\in\mathbb{U}_0$.
This will be the case when the loss function $\ell$ is not too complex
(cf.\ section \ref{sec:verifying}).
For example, in the extreme case where the decision space $U$ is finite,
uniformity is automatic, and thus Theorem \ref{thm:woptfinite} in the
introduction follows immediately from Theorem \ref{thm:wkopt}.

Recall that a pathwise optimal strategy is necessarily weakly pathwise 
optimal.  This is reflected, for example, in Theorems \ref{thm:soptfinite}
and \ref{thm:woptfinite}: indeed, note that 
\begin{align*}
	&
	\|\mathbf{P}[A\cap T^kB|\mathcal{Z}]-
	\mathbf{P}[A|\mathcal{Z}]\,\mathbf{P}[T^kB|\mathcal{Z}]\|_1
	\\
	&=
	\|\mathbf{E}[\{\mathbf{1}_A-\mathbf{P}[A|\mathcal{Z}]\}\,
	\mathbf{1}_{T^kB}|\mathcal{Z}]\|_1
	\\
	&\le
	\|\mathbf{E}[\{\mathbf{1}_A-\mathbf{P}[A|\mathcal{Z}]\}\,
	\mathbf{P}[T^kB|T^{k-n}\mathcal{X}]|\mathcal{Z}]\|_1
	+
	\|\mathbf{1}_{T^kB}-\mathbf{P}[T^kB|T^{k-n}\mathcal{X}]\|_1
	\\
	&\le
	\|\mathbf{P}[A|\mathcal{Z}\vee T^{k-n}\mathcal{X}]
	-\mathbf{P}[A|\mathcal{Z}]\|_1
	+
	\|\mathbf{1}_{B}-\mathbf{P}[B|T^{-n}\mathcal{X}]\|_1
\end{align*}
for any $n,k$, so that the conditional $K$-property implies the 
conditional weak mixing property (relative to any $\sigma$-field 
$\mathcal{Z}$) by letting $k\to\infty$, then $n\to\infty$.  Along the 
same lines, one can show that a slight variation of the assumptions of 
Theorem \ref{thm:sopt} imply the assumption of Theorem \ref{thm:wkopt} 
(modulo minor issues of truncation, which could have been absorbed in 
Theorem \ref{thm:sopt} also at the expense of heavier notation).  It is 
not entirely obvious, at first sight, how far apart the conclusions of our main 
results really are.  For example, in the setting of full information, 
cf.\ Example \ref{ex:fullinf}, the assumption 
of Theorem \ref{thm:wkopt} holds automatically (as then 
$\bar\ell_0^M(u)\circ T^{-k}$ is $\mathcal{Y}_{-\infty,0}$-measurable 
for every $u\in\mathbb{U}_0$ and $k\ge 1$).  Moreover, the reader can 
easily verify that in all the examples we have given where no pathwise 
optimal strategy exists (Examples \ref{ex:merhav}, \ref{ex:filt}, 
\ref{ex:asp2}, \ref{ex:asp2ctd}), even the existence of a weakly 
pathwise optimal strategy fails.  It is therefore tempting to assume
that in a typical situation where a weakly pathwise optimal
strategy exists, there will likely also be a pathwise optimal strategy.
The following example, which is a manifestation of a rather surprising
result in ergodic theory due to Conze \cite{Con73},
provides some evidence to the contrary.  

\begin{example}[Generic transformations]
\label{ex:conze}
In this example, we fix the probability space $(\Omega,\mathcal{B},
\mathbf{P})$, where $\Omega=[0,1]$ with its Borel
$\sigma$-field $\mathcal{B}$ and the Lebesgue measure $\mathbf{P}$.
We will consider the decision space $U=\{0,1\}$ and loss function
$\ell$ defined as
$$
	\ell(u,\omega) = -u\,(\mathbf{1}_{[0,1/2]}(\omega)-1/2)
	\qquad\mbox{for }(u,\omega)\in U\times\Omega.
$$
Moreover, we will consider the setting of blind decisions,
that is, $\mathcal{Y}$ is trivial.

We have not yet defined a transformation $T$.  Our aim is to prove the 
following: \emph{for a generic invertible measure-preserving 
transformation $T$, there is a mean-optimal strategy $\mathbf{\tilde u}$ 
that is weakly pathwise optimal but not pathwise optimal}.  This 
shows not only that there can be a substantial gap between Theorems 
\ref{thm:soptfinite} and \ref{thm:woptfinite}, but that this is in fact the 
typical situation (at least in the sense of weak topology).

Let us recall some basic notions.  Denote by $\mathscr{T}$ the set
of all invertible measure-preserving transformations of 
$(\Omega,\mathcal{B},\mathbf{P})$.  The weak topology on $\mathscr{T}$
is the topology generated by the basic neighborhoods
$B(T_0,B,\varepsilon)=\{T\in\mathscr{T}:\mathbf{P}[TB\mathop{\triangle}T_0B]<
\varepsilon\}$ for all $T_0\in\mathscr{T}$, $B\in\mathcal{B}$,
$\varepsilon>0$.  A property is said to hold for a \emph{generic}
transformation if it holds for every transformation $T$ in a dense
$G_\delta$ subset of $\mathscr{T}$.  A well-known result of Halmos
\cite{Hal44} states that a generic transformation is weak mixing.
Therefore, for a generic transformation, any mean-optimal strategy
$\mathbf{\tilde u}$ is weakly pathwise optimal
by Theorem \ref{thm:woptfinite}.  This proves the first
part of our statement.

Of course, in the present setting, 
$\mathbf{E}[\ell_k(u)|\mathcal{Y}_{0,k}]= \mathbf{E}[\ell_k(u)]=0$ for 
every decision $u\in U$.  Therefore, \emph{every} admissible strategy 
$\mathbf{u}$ is mean-optimal, and the optimal mean loss is given by 
$L^\star=0$, regardless of the choice of transformation 
$T\in\mathscr{T}$.  It is natural to choose a stationary strategy 
$\mathbf{\tilde u}$ (for example, $\tilde u_k=1$ for all $k$) so that 
$\lim_{T\to\infty}L_T(\mathbf{\tilde u})=L^\star$ a.s.  We will show 
that for a generic transformation, the strategy $\mathbf{\tilde u}$ is 
not pathwise optimal.  To this end, it evidently suffices to find 
another strategy $\mathbf{u}$ such that 
$\liminf_{T\to\infty}L_T(\mathbf{u})<L^\star$ with positive probability. 

To this end, we use the following result of Conze that can be read off 
from the proof of \cite[Theorem 5]{Con73}: there exists a sequence 
$n_k\uparrow\infty$ with $k/n_k\to 1/2$ such that 
for every $0<\alpha<1$ and $1/2<\lambda<1$, a generic transformation $T$
satisfies
$$
	\mathbf{P}\bigg[
	\limsup_{N\to\infty}
	\frac{1}{N}\sum_{k=1}^N \mathbf{1}_{[0,1/2]}\circ T^{n_k}
	\ge \lambda
	\bigg]\ge 1-\alpha.
$$
Define the strategy $\mathbf{u}$ such that
$u_n=1$ if $n=n_k$ for some $k$, and $u_n=0$ otherwise.  Then,
for a generic transformation $T$, we have with probability
at least $1-\alpha$
$$
	\liminf_{T\to\infty}L_{n_T}(\mathbf{u}) =
	-
	\limsup_{T\to\infty}
	\frac{1}{n_T}\sum_{k=1}^T
	(\mathbf{1}_{[0,1/2]}\circ T^{n_k}-1/2) 
	\le -\frac{2\lambda-1}{4}.
$$
In words, we have shown that for a generic transformation $T$, the
time-average loss of the mean-optimal strategy $\mathbf{\tilde u}$ 
exceeds that of the strategy $\mathbf{u}$ infinitely often by
almost $1/4$ with almost unit probability.  Thus the mean-optimal
strategy $\mathbf{\tilde u}$ fails to be pathwise optimal in a very
strong sense, and our claim is established.
\end{example}

Example \ref{ex:conze} only shows that there is a mean-optimal strategy 
$\mathbf{\tilde u}$ that is weakly pathwise optimal but not pathwise 
optimal. It does not make any statement about whether or not a pathwise 
optimal strategy $\mathbf{u}^\star$ actually exists.  However, we do not 
know of any mechanism that might lead to pathwise optimality in such a 
setting. We therefore conjecture that for a generic transformation a 
pathwise optimal strategy in fact fails to exist at all, so that (unlike 
in the full information setting) pathwise optimality and weak pathwise 
optimality are really distinct notions.

The result of Conze used in Example \ref{ex:conze} originates from a 
deep problem in ergodic theory that aims to understand the validity of 
individual ergodic theorems for subsequences, cf.\ \cite{Con73,BL85} and 
the references therein. A general characterization of such ergodic 
properties does not appear to exist, which suggests that the pathwise 
optimality property may be difficult to characterize beyond general 
sufficient conditions such as Theorem \ref{thm:sopt}.  In contrast, the
weak pathwise optimality property is much more regularly behaved.
The following theorem, which will be proved in section \ref{sec:proofwkmixcv}
below, provides a complete characterization of weak pathwise optimality
in the special case that the observation field $\mathcal{Y}$
is invariant.

\begin{theorem}
\label{thm:wkmixcv}
Let $(\Omega,\mathcal{B},\mathbf{P},T)$ be an ergodic dynamical system,
and suppose that $(\Omega,\mathcal{B},\mathbf{P})$ is a standard probability
space and that $\mathcal{Y}\subseteq\mathcal{B}$ is an invariant $\sigma$-field
(that is, $\mathcal{Y}=T^{-1}\mathcal{Y}$).  Then
the following are equivalent:
\begin{enumerate}
\item $(\Omega,\mathcal{B},\mathbf{P},T)$ conditionally weak mixing relative to
$\mathcal{Y}$.
\item For every bounded loss function $\ell:U\times\Omega\to\mathbb{R}$ with
finite decision space $\card U<\infty$, there exists a weakly pathwise
optimal strategy.
\end{enumerate}
\end{theorem}

The invariance of $\mathcal{Y}$ is automatic in the setting 
of blind decisions (as $\mathcal{Y}$ is trivial), in which 
case Theorem \ref{thm:wkmixcv} yields a decision-theoretic 
characterization of the weak mixing property. In more general 
observation models, invariance of $\mathcal{Y}$ may be an 
unnatural requirement from the point of view of decisions under partial 
information, as it implies that there is no information gain over time.  
On the other hand, applications of the notion of conditional weak mixing 
relative to a $\sigma$-field $\mathcal{Z}$ in ergodic theory almost 
always assume that $\mathcal{Z}$ is invariant (e.g., \cite{Rud04}). 
Theorem \ref{thm:wkmixcv} yields a decision-theoretic interpretation of 
this property by choosing $\mathcal{Y}=\mathcal{Z}$.

\section{Complexity and conditional ergodicity}
\label{sec:verifying}

\subsection{Universal complexity assumptions}
\label{sec:unicplx}

The goal of this section is to develop complexity assumptions on the 
loss function $\ell$ that ensure that the uniform convergence 
assumptions in our main results hold regardless of any properties of the 
transformation $T$ or observations $\mathcal{Y}$.  While such universal 
complexity assumptions are not always necessary (for example, in the 
full information setting uniform convergence holds regardless of the 
loss function), they frequently hold in practice and provide easily 
verifiable conditions that ensure that our results hold in a broad class
of decision problems with partial information.

The simplest assumption is Grothendieck's notion of equimeasurability 
\cite{Gro55}.

\begin{definition}
\label{def:equimeas}
The loss function $\ell:U\times\Omega\to\mathbb{R}$ on the probability
space $(\Omega,\mathcal{B},\mathbf{P})$ is said to be \emph{equimeasurable}
if for every $\varepsilon>0$, there exists $\Omega_\varepsilon\in\mathcal{B}$
with $\mathbf{P}[\Omega_\varepsilon]\ge 1-\varepsilon$
such that the class of functions 
$\{\ell_0(u)\mathbf{1}_{\Omega_\varepsilon}:u\in U\}$ 
is totally bounded in $L^\infty(\mathbf{P})$.
\end{definition}

The beauty of this simple notion is that it ensures uniform convergence
of almost anything.  In particular, we obtain the following results.

\begin{lemma}
\label{lem:equi}
Suppose that the loss function $\ell$ is equimeasurable.  Then
Assumption 2 of Corollary \ref{cor:smeanopt} holds, and thus
Assumption 2 of Theorem \ref{thm:sopt} holds as well, provided
that $\mathcal{X}$ is a generating $\sigma$-field (that is,
$\bigvee_n T^{-n}\mathcal{X}=\mathcal{B}$).
\end{lemma}

\begin{proof}
Let us establish the first line of Assumption 2.  Fix $\varepsilon>0$
and $\Omega_\varepsilon$ as in Definition \ref{def:equimeas}.
Then there exist $N<\infty$ measurable functions 
$l_1,\ldots,l_N:\Omega\to\mathbb{R}$ such that for every
$u\in U$, there exists $k(u)\in\{1,\ldots,N\}$ such that
$$
	\|\ell_0(u)\mathbf{1}_{\Omega_\varepsilon}-
	l_{k(u)}\mathbf{1}_{\Omega_\varepsilon}\|_\infty\le\varepsilon
$$
(and $u\mapsto k(u)$ can clearly be chosen to be measurable).
It follows that
\begin{align*}
	\esssup_{u\in\mathbb{U}_0}
	\big|
	\mathbf{E}[\ell_0(u)|\mathcal{Y}_{-\infty,0}\vee T^{-n}\mathcal{X}]
	-\ell_0(u)
	\big| \le
	&\max_{1\le k\le N}
	\big|
	\mathbf{E}[l_{k}\mathbf{1}_{\Omega_\varepsilon}
	|\mathcal{Y}_{-\infty,0}\vee T^{-n}\mathcal{X}]
	-l_{k}\mathbf{1}_{\Omega_\varepsilon}
	\big|  \\ 
	&+2\varepsilon
	+\mathbf{E}[\Lambda\mathbf{1}_{\Omega_\varepsilon^c}
	|\mathcal{Y}_{-\infty,0}\vee T^{-n}\mathcal{X}]
	+\Lambda\mathbf{1}_{\Omega_\varepsilon^c}.
\end{align*}
As $\mathcal{X}$ is generating, the martingale convergence theorem gives
$$
	\limsup_{n\to\infty}
	\bigg\|
	\esssup_{u\in\mathbb{U}_0}
	\big|
	\mathbf{E}[\ell_0(u)|\mathcal{Y}_{-\infty,0}\vee T^{-n}\mathcal{X}]
	-\ell_0(u)
	\big|\bigg\|_1 \le
	2\varepsilon+\mathbf{E}[2\Lambda\mathbf{1}_{\Omega_\varepsilon^c}].
$$
Letting $\varepsilon\downarrow 0$ yields the first line of Assumption 2.  
The remaining statements of Assumption 2 follow by an 
essentially identical argument.
\qed\end{proof}

\begin{lemma}
Suppose that the following conditional mixing assumption holds:
$$
	\lim_{M\to\infty}
	\limsup_{T\to\infty}
	\bigg\|
	\frac{1}{T}\sum_{k=1}^T
	|\mathbf{E}[\{\bar\ell_0^M(u)\circ T^{-k}\} ~\bar\ell_0^M(u')|
	\mathcal{Y}_{-\infty,0}]|
	\bigg\|_1 = 0
	\quad\mbox{for every }u,u'\in U.
$$
If the loss function $\ell$ is equimeasurable, then the
assumption of Theorem \ref{thm:wkopt} holds.
\end{lemma}

\begin{proof}
The proof is very similar to that of Lemma \ref{lem:equi} and is therefore
omitted.
\qed\end{proof}

As an immediate consequence of these lemmas, we have:

\begin{corollary}
\label{cor:equiisgreat}
The conclusions of Theorems \ref{thm:soptfinite} and \ref{thm:woptfinite}
remain in force if the assumption that $U$ is finite is replaced by the
assumption that $\ell$ is equimeasurable.
\end{corollary}

We now give a simple condition for equimeasurability
that suffices in many cases. It is closely
related to a result of Mokobodzki (cf.\ \cite[Theorem IX.19]{DM88}).

\begin{lemma}
\label{lem:moko}
Suppose that $U$ is a compact metric space and that
$u\mapsto\ell(u,\omega)$ is continuous for a.e.\ $\omega\in\Omega$.
Then $\ell$ is equimeasurable.
\end{lemma}

\begin{proof}
As $U$ is a compact metric space (with metric $d$), 
it is certainly separable.  Let
$U_0\subseteq U$ be a countable dense set, and define the functions
$$
	b_n = \sup_{u,u'\in U_0:d(u,u')\le n^{-1}}|\ell_0(u)-\ell_0(u')|.
$$
$b_n$ is measurable, as it is the supremum of countably many random
variables.  Moreover,
for almost every $\omega$, the function $u\mapsto\ell(u,\omega)$ is
uniformly continuous (being continuous on a compact metric space).  Therefore,
$b_n\downarrow 0$ a.s.\ as $n\to\infty$.

By Egorov's theorem, there exists for every $\varepsilon>0$ a set
$\Omega_\varepsilon$ with $\mathbf{P}[\Omega_\varepsilon]\ge 1-\varepsilon$
such that $\|b_n\mathbf{1}_{\Omega_\varepsilon}\|_\infty\downarrow 0$.
We claim that $\{\ell_0(u)\mathbf{1}_{\Omega_\varepsilon}:u\in U\}$ is
compact in $L^\infty$.  Indeed, for any sequence $(u_n)_{n\ge 1}\subseteq U$
we may choose a subsequence $(u_{n_k})_{k\ge 1}$ that converges to 
$u_\infty\in U$.  Then for every $r$, we have
$|\ell_0(u_{n_k})-\ell_0(u_\infty)|\le b_r$
for all $k$ sufficiently large, and therefore
$\|\ell_0(u_{n_k})\mathbf{1}_{\Omega_\varepsilon}-\ell_0(u_\infty)
\mathbf{1}_{\Omega_\varepsilon}\|_\infty\to 0$.
\qed\end{proof}

Let us give two standard examples of decision problems (cf.\ \cite{Alg94,Nob03}).

\begin{example}[$\ell_p$-prediction]
\label{ex:meansq}
Consider the stochastic process setting $(X,Y)$, and let $f$ be a bounded
function.  The aim is, at each time $k$, to choose a predictor $u_k$ of 
$f(X_{k+1})$ on the basis of the observation history $Y_0,\ldots,Y_k$.  
We aim to minimize the pathwise time-average $\ell_p$-prediction loss 
$\frac{1}{T}\sum_{k=1}^T |u_k-f(X_{k+1})|^p$ ($p\ge 1$). 
This is a particular decision problem with partial information, where 
the loss function is given by $\ell_0(u) = |u-f(X_1)|^p$ and the decision 
space is $U=[\inf_xf(x),\sup_xf(x)]$.  It is immediate that $\ell$
is equimeasurable by Lemma \ref{lem:moko}.
\end{example}

\begin{example}[Log-optimal portfolios]
Consider a market with $d$ securities (e.g., $d-1$ stocks and one bond)
whose returns in day $k$ are given by the random variable $X_k$ with values 
in $\mathbb{R}_+^d$. The decision space $U=\{p\in\mathbb{R}_+^{d}:
\sum_{i=1}^dp_i=1\}$ is the simplex: $u_k^i$ represents the fraction of
wealth invested in the $i$th security in day $k$.  The total wealth at time
$T$ is therefore given by $\prod_{k=1}^T\langle u_k,X_k\rangle$.
We only have access to partial information $Y_k$ in day $k$, e.g., from
news reports. We aim
to choose an investment strategy on the basis of the available information
that maximizes the wealth, or, equivalently,
its growth $\frac{1}{T}\sum_{k=1}^T \log\langle u_k,X_k\rangle$.
This corresponds to a decision problem with partial information for 
the loss function $\ell_0(u)=-\log\langle u,X_0\rangle$.

In order for the loss to be dominated in $L^1$, we impose the mild 
assumption $\mathbf{E}[\Lambda]<\infty$ with $\Lambda = 
\sum_{i=1}^d|\log X_0^i|$. We claim that the loss $\ell$ is then also 
equimeasurable.  Indeed, as $\mathbf{E}[\Lambda]<\infty$, the returns 
must satisfy $X_0^i>0$ a.s.\ for every $i$.  Therefore, equimeasurability
follows directly from Lemma \ref{lem:moko}.
\end{example}

As we have seen above, equimeasurability follows easily when the loss 
function possesses some mild pointwise continuity properties. However, 
there are situations when this may not be the case.  In particular, 
suppose that $\ell(u,\omega)$ only takes the values $0$ and $1$, that 
is, our decisions are sets (as may be the case, for example, in 
predicting the shape of an oil spill or in sequential classification 
problems).  In such a case, equimeasurability will rarely hold, and it 
is of interest to investigate alternative complexity assumptions.  As we 
will presently explain, equimeasurability is almost necessary to obtain 
a universal complexity assumption for Corollary \ref{cor:smeanopt}; 
however, in the setting of Theorem \ref{thm:sopt}, the assumption can
be weakened considerably.

The simplicity of the equimeasurability assumption hides the fact that 
there are two distinct uniformity assumptions in Corollary 
\ref{cor:smeanopt}: we require uniform convergence of both martingales 
and reverse martingales, which are quite distinct phenomena (cf.\ 
\cite{HJ90,vH12ugc}).  The uniform convergence of 
martingales can be restrictive.

\begin{example}[Uniform martingale convergence]
Let $(\mathcal{G}_n)_{n\ge 1}$ be a filtration such that each
$\mathcal{G}_n=\sigma\{\pi_n\}$ is generated by a finite measurable 
partition $\pi_n$ of the probability space $(\Omega,\mathcal{B},\mathbf{P})$.
Let $L:\mathbb{N}\times\Omega\to\mathbb{R}$ a bounded function
such that
$L(u,\cdot\,)$ is $\mathcal{G}_\infty$-measurable for every $u\in\mathbb{N}$.
Then $\mathbf{E}[L(u,\cdot\,)|\mathcal{G}_n]\to L(u,\cdot\,)$ a.s.\ for every
$u$.  We claim that if this martingale convergence is in fact uniform, that is,
$$
	\sup_{u\in\mathbb{N}}|\mathbf{E}[L(u,\cdot\,)|\mathcal{G}_n]-
	L(u,\cdot\,)|\xrightarrow{n\to\infty}0\quad\mbox{in }L^1,
$$
then $L$ must necessarily be equimeasurable.  
To see this, let us first extract a subsequence $n_k\uparrow\infty$
along which the uniform martingale convergence holds a.s.  
Fix $\varepsilon>0$.  By Egorov's theorem, there exists a set
$\Omega_\varepsilon$ with $\mathbf{P}[\Omega_{\varepsilon}]\ge 1-\varepsilon$
such that
$$
	\sup_{u\in\mathbb{N}}
	\|\mathbf{E}[L(u,\cdot\,)|\mathcal{G}_{n_k}]
	\mathbf{1}_{\Omega_\varepsilon}-
	L(u,\cdot\,)\mathbf{1}_{\Omega_\varepsilon}
	\|_\infty\xrightarrow{k\to\infty}0.
$$
Therefore,
for every $\alpha>0$, there exists $k$ such that
$$
	\sup_{u\in\mathbb{N}}
	\|\alpha\lfloor\alpha^{-1}\mathbf{E}[L(u,\cdot\,)|\mathcal{G}_{n_k}]
	\mathbf{1}_{\Omega_\varepsilon}\rfloor-
	L(u,\cdot\,)\mathbf{1}_{\Omega_\varepsilon}
	\|_\infty\le 2\alpha.
$$
But as $\mathcal{G}_n$ is finitely generated, we can write
$$
	\mathbf{E}[L(u,\cdot\,)|\mathcal{G}_n]\mathbf{1}_{\Omega_\varepsilon} =
	\sum_{P\in\pi_n}L_{n,u,P}\mathbf{1}_{P\cap\Omega_\varepsilon}, 
$$ 
with $|L_{n,u,P}|\le \|L\|_\infty$ for all $n,u,P$.  In particular, 
$\{\alpha\lfloor\alpha^{-1}\mathbf{E}[L(u,\cdot\,)|\mathcal{G}_{n}] 
\mathbf{1}_{\Omega_\varepsilon}\rfloor:u\in\mathbb{N}\}$ is a finite 
family of random variables for every $n$.  We have therefore established 
that the family 
$\{L(u,\cdot\,)\mathbf{1}_{\Omega_\varepsilon}:u\in\mathbb{N}\}$ is 
totally bounded in $L^\infty$. 
\end{example}

In the context of Corollary \ref{cor:smeanopt}, the previous example can 
be interpreted as follows.  Suppose that the observations are 
finite-valued, that is, $\mathcal{Y}$ is a finitely generated 
$\sigma$-field. Let us suppose, for simplicity, that the decision space 
$U$ is countable (the same conclusion holds for general $U$ modulo some 
measurability issues). Then, if the third line of Assumption 2 in 
Corollary \ref{cor:smeanopt} holds, then the conditioned loss 
$\mathbf{E}[\ell_0(u)|\mathcal{Y}_{-\infty,0}]$ is necessarily 
equimeasurable.  While it is possible that the conditioned loss is 
equimeasurable even when the loss $\ell$ is not (e.g., in the case of 
blind decisions), this is somewhat unlikely to be the case given a 
nontrivial observation structure.  Therefore, it appears that 
equimeasurability is almost necessary to obtain universal complexity 
assumptions in the setting of Corollary \ref{cor:smeanopt}.

The situation is much better in the setting of Theorem \ref{thm:sopt},
however.  While the first line of Assumption 2 in Theorem \ref{thm:sopt}
is still a uniform martingale convergence property, the $\sigma$-field
$\mathcal{X}$ cannot be finitely generated except in trivial cases.
In fact, in many cases the loss $\ell$ will be $T^{-n}\mathcal{X}$-measurable
for some $n<\infty$, in which case the first line of Assumption 2 is
automatically satisfied (in particular, in the stochastic process setting, 
this will be the case for \emph{finitary} loss $\ell_0(u)=
l(u,X_{n_1},\ldots,X_{n_k})$ if we choose 
$\mathcal{X}=\sigma\{X_k,Y_k:k\le 0\}$).  The remainder of Assumption 2
is a uniform reverse martingale convergence property, which holds
under much weaker assumptions.

\begin{definition}
The loss $\ell:U\times\Omega\to\mathbb{R}$ on
$(\Omega,\mathcal{B})$ is said to be \emph{universally bracketing} if
for every probability measure $\mathbf{P}$ and $\varepsilon,M>0$, the 
family $\{\ell_0(u)\mathbf{1}_{\Lambda\le M}:u\in U\}$ can be covered by finitely many brackets
$\{f:g\le f\le h\}$ with $\|g-h\|_{L^1(\mathbf{P})}\le\varepsilon$.
\end{definition}

\begin{lemma}
Let $(\Omega,\mathcal{B})$ be a standard space, and let
$\mathcal{X},\mathcal{Y}$ be countably generated.
Suppose the loss $\ell$ is universally bracketing and finitary 
(that is, for some $n\in\mathbb{Z}$, $\ell_0(u)$ is 
$T^{-n}\mathcal{X}$-measurable for all $u\in U$).
Then Assumption 2 of Theorem \ref{thm:sopt} holds.
\end{lemma}

\begin{proof}
The finitary assumption trivially implies the first line of Assumption 2.
The second line follows along the lines of the proof of
\cite[Corollary 1.4(2$\Rightarrow$7)]{vH12ugc}.\footnote{
	The pointwise separability assumption in 
	\cite[Corollary 1.4(2$\Rightarrow$7)]{vH12ugc} is
	not needed here, as the essential supremum can be reduced
	to a countable supremum as in the proof of Lemma \ref{lem:meanopt}.
}
\qed\end{proof}

To show that universal bracketing can be much weaker than equimeasurability,
we give a simple example in the context of set estimation.

\begin{example}[Confidence intervals]
Consider the stochastic process setting $(X,Y)$ where $X$ takes values 
in the set $[-1,1]$, and fix $\varepsilon>0$.  We would like to pin down 
the value of $X_k$ up to precision $\varepsilon$; that is, we want to 
choose $u_k\in[-1,1]$ as a function of the observations $Y_0,\ldots,Y_k$ 
such that $u_k\le X_k<u_k+\varepsilon$ as often as possible.  This
is a partial information decision problem with loss function
$\ell_0(u)=\mathbf{1}_{\mathbb{R}\backslash[u,u+\varepsilon[}(X_0)$.

The proof of the universal bracketing property of $\ell$ is standard.
Given $\mathbf{P}$ and $\varepsilon>0$, we choose
$-1=a_0<a_1<\cdots<a_n=1$ (for some finite $n$) in such a way that
$\mathbf{P}[a_i<X_0<a_{i+1}]\le\varepsilon$ for all $i$ (note that
every atom of $X_0$ with probability greater than $\varepsilon$ is
one of the values $a_i$).  Put each function $\ell_0(u)$ such that
$u=a_i$ or $u+\varepsilon=a_i$ for some $i$ in its own bracket, and
consider the additional brackets 
$\{f:\mathbf{1}_{\mathbb{R}\backslash]a_{i-1},a_{j+1}[}\le f\le
\mathbf{1}_{\mathbb{R}\backslash[a_i,a_j]}\}$ for all $1\le i\le j<n$.
Then evidently each of the brackets has diameter not exceeding
$2\varepsilon$, and for every $u\in U$ the function $\ell_0(u)$ is 
included in one of the brackets thus constructed.

On the other hand, whenever the law of $X_0$ is not purely atomic, the 
loss $\ell$ cannot be equimeasurable. Indeed, as 
$\|\ell_0(u)\mathbf{1}_{\Omega_\varepsilon}-\ell_0(u') 
\mathbf{1}_{\Omega_\varepsilon}\|_\infty=1$ whenever 
$\ell_0(u)\mathbf{1}_{\Omega_\varepsilon}\ne\ell_0(u') 
\mathbf{1}_{\Omega_\varepsilon}$, it is impossible for 
$\{\ell_0(u)\mathbf{1}_{\Omega_\varepsilon}:u\in U\}$ to be totally 
bounded in $L^\infty$ for any infinite set $\Omega_\varepsilon$ (and 
therefore for any set of sufficiently large measure). 
\end{example}

In \cite{vH12ugc} a detailed characterization is given of the universal 
bracketing property.  In particular, it is shown that a uniformly 
bounded, separable loss $\ell$ on a standard measurable space 
is universally bracketing if and only if $\{\ell_0(u):u\in U\}$ is a 
universal Glivenko-Cantelli class, that is, a class of functions for 
which the law of large numbers always holds uniformly. Many useful 
methods have been developed in empirical process theory to verify this 
property, cf.\ \cite{Dud99,VW96}.  In particular, for a separable 
$\{0,1\}$-valued loss, a very useful sufficient condition is that 
$\{\ell_0(u):u\in U\}$ is a Vapnik-Chervonenkis class.  We refer to
\cite{vH12ugc,Dud99,VW96} for further details.

\subsection{Conditional absolute regularity}

In the previous section, we have developed universal complexity 
assumptions that are applicable regardless of other details of the 
model. In the present section, we will in some sense take the opposite 
approach: we will develop a sufficient condition for a stronger version 
of the conditional $K$-property (in the stochastic process setting) 
under which no complexity assumptions are needed.  This shows that there 
is a tradeoff between mixing and complexity; if the mixing assumption is 
strengthened, then the complexity assumption can be weakened.  
An additional advantage of the sufficient condition to be presented is that
it is in practice one of the most easily verifiable conditions that ensures the
conditional $K$-property.

In the remainder of this section, we will work in the stochastic process 
setting. Let $(X,Y)$ be a stationary ergodic process taking values in 
the Polish space $E\times F$.  We define
$\mathcal{Y}_{n,m}=\sigma\{Y_k:n\le k\le m\}$ and
$\mathcal{X}_{n,m}=\sigma\{X_k:n\le k\le m\}$ for $n\le m$, and we 
consider the observation and generating fields
$\mathcal{Y}=\sigma\{Y_0\}$, $\mathcal{X}=\mathcal{X}_{-\infty,0}\vee
\mathcal{Y}_{-\infty,0}$.  In this setting, the conditional $K$-property
relative to $\mathcal{Y}_{-\infty,0}$ reduces to
$$
	\bigcap_{k=1}^\infty(\mathcal{Y}_{-\infty,0}\vee
	\mathcal{X}_{-\infty,-k})=\mathcal{Y}_{-\infty,0}
	\quad\mathop{\mathrm{mod}}\mathbf{P}.
$$
If $\mathcal{Y}$ is trivial (that is, the observations $Y$ are noninformative),
this reduces to the statement that $X$ has a trivial past tail $\sigma$-field,
that is, $X$ is regular (or purely nondeterministic) in the sense of
Kolmogorov.  This property is often fairly easy to check: for example, any 
Markov chain whose law converges weakly to a unique invariant measure is regular
(cf.\ \cite[Prop.\ 3]{Tot70}).  When $\mathcal{Y}$ is nontrivial, the 
conditional $K$-property is generally not so easy to check, however.
We therefore give a condition, arising from filtering theory \cite{TvH12},
that allows to deduce conditional mixing properties from their
more easily verifiable unconditional counterparts.

We will require two assumptions.  The first assumption states that the 
pair $(X,Y)$ is absolutely regular in the sense of Volkonski{\u\i} and 
Rozanov \cite{VR59} (this property is also known as $\beta$-mixing).
Absolute regularity is a strengthening of the regularity property;
assuming regularity of $(X,Y)$ is not sufficient for what follows
\cite{vH12}.  Many techniques have been developed to verify the
absolute regularity property; for example, any Harris recurrent
and aperiodic Markov chain is absolutely regular \cite{MT09}.

\begin{definition}
\label{def:absreg}
The process $(X,Y)$ is said to be \emph{absolutely regular} if
$$
	\big\|
	\mathbf{P}[(X_k,Y_k)_{k\ge n}\in\cdot\,|
	\mathcal{X}_{-\infty,0}\vee\mathcal{Y}_{-\infty,0}]
	- \mathbf{P}[(X_k,Y_k)_{k\ge n}\in\cdot\,]
	\big\|_{\rm TV}\xrightarrow{n\to\infty}0\quad
	\mbox{in }L^1.
$$
\end{definition}

By itself, however, absolute regularity of $(X,Y)$ is not sufficient for 
the conditional $K$-property, as can be seen in Example \ref{ex:filt}.
In this example, the relation between the processes $X$ and $Y$ is very
singular, so that things go wrong when we condition.  The following
nondegeneracy assumption rules out this possibility.

\begin{definition}
\label{def:nondeg}
The process $(X,Y)$ is said to be \emph{nondegenerate} if
$$
	\mathbf{P}[Y_1,\ldots,Y_m\in\cdot\,|
	\mathcal{Z}_{-\infty,0}\vee
	\mathcal{Z}_{m+1,\infty}]
	\sim
	\mathbf{P}[Y_1,\ldots,Y_m\in\cdot\,|
	\mathcal{Y}_{-\infty,0}\vee
	\mathcal{Y}_{m+1,\infty}]
	\quad \mbox{a.s.}
$$
for every $1\le m<\infty$, where $\mathcal{Z}_{n,m}:=
\mathcal{X}_{n,m}\vee\mathcal{Y}_{n,m}$.
\end{definition}

The nondegeneracy assumption ensures that the null sets of the 
law of the observations $Y$ do not depend too much on the unobserved
process $X$.  The assumption is often easily verified.  For example,
if $Y_k=f(X_k)+\eta_k$ where $\eta_k$ is an i.i.d.\ sequence of random
variables with
strictly positive density, then the conditional distributions
in Definition \ref{def:nondeg} have strictly positive densities and
are therefore equivalent a.s.

\begin{theorem}[\cite{TvH12}]
\label{thm:infdim}
If $(X,Y)$ is absolutely regular and nondegenerate, then
$$
	\bigcap_{k=1}^\infty(\mathcal{Y}_{-\infty,0}\vee
	\mathcal{X}_{-\infty,-k})=\mathcal{Y}_{-\infty,0}
	\quad\mathop{\mathrm{mod}}\mathbf{P}.
$$
\end{theorem}

Theorem \ref{thm:infdim} provides a practical method to check the 
conditional $K$-property.  However, the proof of Theorem 
\ref{thm:infdim} actually yields a much stronger statement. It 
is shown in \cite[Theorem 3.5]{TvH12} that if $(X,Y)$ is absolutely 
regular and nondegenerate, then $X$ is \emph{conditionally} absolutely 
regular relative to $\mathcal{Y}_{-\infty,\infty}$ in the sense that
$$
	\big\|
	\mathbf{P}[(X_k)_{k\ge n}\in\cdot\,|
	\mathcal{X}_{-\infty,0}\vee\mathcal{Y}_{-\infty,\infty}]
	- \mathbf{P}[(X_k)_{k\ge n}\in\cdot\,|
	\mathcal{Y}_{-\infty,\infty}]
	\big\|_{\rm TV}\xrightarrow{n\to\infty}0\quad
	\mbox{in }L^1.
$$
Moreover, it is shown\footnote{
	Some of the statements in \cite{TvH12} are time-reversed as 
	compared to their counterparts stated here.  However, as both the 
	absolute regularity and the nondegeneracy assumptions are 
	invariant under time reversal (cf.\ \cite{VR59} for the
	former; the latter is trivial), the present statements 
	follow immediately.
} that under the same assumptions \cite[Proposition 3.9]{TvH12}
$$
	\mathbf{P}[(X_k)_{k\le 0}\in\cdot\,|\mathcal{Y}_{-\infty,0}]
	\sim
	\mathbf{P}[(X_k)_{k\le 0}\in\cdot\,|\mathcal{Y}_{-\infty,\infty}]
	\quad\mbox{a.s.}
$$
From these properties, we can deduce the following result.

\begin{theorem}
\label{thm:soptinfdim}
In the setting of the present section, 
suppose that $(X,Y)$ is absolutely regular and nondegenerate, and consider
a loss function of the form $\ell_0(u) = l(u,X_0)$.  Then the conclusions
of Theorem \ref{thm:sopt} hold.
\end{theorem}

The key point about Theorem \ref{thm:soptinfdim} is that no complexity 
assumption is imposed: the loss function $l(u,x)$ may be an arbitrary 
measurable function (as long as it is dominated in $L^1$ in accordance 
with our standing assumption).  The explanation for this is that the 
conditional absolute regularity property is so strong that the regular 
conditional probabilities 
$\mathbf{P}[X_0\in\cdot\,|\mathcal{Y}_{-\infty,\infty}\vee 
\mathcal{X}_{-\infty,-n}]$ converge in total variation.  Therefore, the 
corresponding reverse martingales converge uniformly over any dominated 
family of measurable functions.  The strength of the conditional mixing 
property therefore eliminates the need for any additional complexity 
assumptions.  In contrast, we may certainly have pathwise optimal strategies
when absolute regularity fails, but then a complexity assumption is essential
(cf.\ Example \ref{ex:asp2}).

The proof of Theorem \ref{thm:soptinfdim} will be given in section 
\ref{sec:proofinfdim}.  The proof is a straightforward adaptation of 
Theorem \ref{thm:sopt}; unfortunately, the fact that the conditional 
absolute regularity property is relative to 
$\mathcal{Y}_{-\infty,\infty}$ rather than $\mathcal{Y}_{-\infty,0}$ 
complicates a direct verification of the assumptions of Theorem 
\ref{thm:sopt} (while this should be possible along the lines of 
\cite{TvH12}, we will follow the simpler route here).  
The results of \cite{TvH12} could also be used to obtain the conclusion of 
Corollary \ref{cor:smeanopt} in the setting of Theorem 
\ref{thm:soptinfdim} under somewhat stronger nondegeneracy assumptions.

\subsection{Hidden Markov models and nonlinear filters}
\label{sec:filt}

The goal of the present section is to explore some implications of our 
results to filtering theory.  For simplicity of exposition, we will 
restrict attention to the classical setting of (general state space) 
hidden Markov models (see, e.g., \cite{CMR05}).

We adopt the stochastic process setting and notations of the previous
section.  In addition, we assume that $(X,Y)$ is a hidden Markov model,
that is, a Markov chain whose transition kernel can be factored as
$\tilde P(x,y,dx',dy')=P(x,dx')\,\Phi(x',dy')$.  This implies that
the process $X$ is a Markov chain in its own right, and that the
observations $Y$ are conditionally independent given $X$.
In the following, we will assume that the observation kernel $\Phi$
has a density, that is, $\Phi(x,dy)=g(x,y)\,\varphi(dy)$ for some
measurable function $g$ and reference measure $\varphi$. 

A fundamental object in this theory is the nonlinear filter $\Pi_k$, 
defined as
$$
	\Pi_k := \mathbf{P}[X_k\in\cdot\,|Y_0,\ldots,Y_k].
$$
The measure-valued process $\Pi=(\Pi_k)_{k\ge 0}$ is itself a
(nonstationary) Markov chain \cite{vH12} with transition kernel 
$\mathscr{P}$.  To study the stationary behavior of the filter, 
which is of substantial interest in applications (see, for example,
\cite{Han09spa} and the references therein),
one must understand the relationship between the ergodic
properties of $X$ and $\Pi$.  
The following result, proved in \cite{vH12}, is essentially
due to Kunita \cite{Kun71}.

\begin{theorem}
\label{thm:kunita}
Suppose that the transition kernel $P$ possesses a unique invariant
measure (that is, $X$ is uniquely ergodic).  Then
the filter transition kernel $\mathscr{P}$ possesses a unique
invariant measure (that is, $\Pi$ is uniquely ergodic)
if and only if
$$
	\bigcap_{k=1}^\infty(\mathcal{Y}_{-\infty,0}\vee
	\mathcal{X}_{-\infty,-k})=\mathcal{Y}_{-\infty,0}
	\quad\mathop{\mathrm{mod}}\mathbf{P}.
$$
\end{theorem}

Evidently, ergodicity of the filter is closely related to the 
conditional $K$-property.  We will exploit this fact to prove a new 
optimality property of nonlinear filters.

The usual interpretation of the filter $\Pi_k$ is that one aims
to track to current location $X_k$ of the unobserved
process on the basis of the observation history $Y_0,\ldots,Y_k$.
By the elementary property of conditional expectations,
$\Pi_k(f)$ provides, for any bounded test function $f$,
an optimal mean-square error estimate of $f(X_k)$:
$$
	\mathbf{E}\big[\{f(X_k)-\Pi_k(f)\}^2\big] \le
	\mathbf{E}\big[\{f(X_k)-\hat f_k(Y_0,\ldots,Y_k)\}^2\big]
	\quad\mbox{for any measurable }\hat f_k.
$$
This interpretation may not be satisfying, however, if only
one sample path of the observations is available (recall Examples
\ref{ex:merhav} and \ref{ex:filt}): one would rather show that
$$
	\liminf_{T\to\infty}
	\Bigg[
	\frac{1}{T}\sum_{k=1}^T
	\{f(X_k)-\hat f_k(Y_0,\ldots,Y_k)\}^2 -
	\frac{1}{T}\sum_{k=1}^T
	\{f(X_k)-\Pi_k(f)\}^2
	\Bigg]\ge 0\quad\mbox{a.s.}
$$
for any alternative sequence of estimators $(\hat f_k)_{k\ge 0}$.
If this property holds for any bounded test function $f$,
the filter will be said to be \emph{pathwise optimal}.

\begin{corollary}
Suppose that the filtering process $\Pi$ is uniquely ergodic.
Then the filter is both mean-square optimal and pathwise optimal.
\end{corollary}

\begin{proof}
Note that the filter $\Pi_k(f)$ is the mean-optimal policy for the
partial information decision problem with loss
$\ell_0(u)=\{f(X_0)-u\}^2$.  As the latter is equimeasurable,
the result follows directly from Theorem \ref{thm:kunita} and
Corollary \ref{cor:smeanopt}.
\qed\end{proof}

The interaction between our main results and the ergodic theory of 
nonlinear filters is therefore twofold.  On the one hand, our main 
results imply that ergodic nonlinear filters are always pathwise 
optimal.  Conversely, Theorem \ref{thm:kunita} shows that ergodicity
of the filter is a sufficient condition for our main results to hold
in the context of hidden Markov models with equimeasurable loss.
This provides another route to establishing the conditional $K$-property:
the filtering literature provides a variety of methods to verify
ergodicity of the filter \cite{vH09,CvH10,CR11,TvH12}.  It should be
noted, however, that ergodicity of the filter is not necessary
for the conditional $K$-property to hold, even in the setting of hidden
Markov models.

\begin{example}
Consider the hidden Markov model $(X,Y)$ where $X$ is the stationary
Markov chain such that $X_0\sim\mathrm{Uniform}([0,1])$ and
$X_{k+1}=2X_k\mathop{\mathrm{mod}}1$, $Y_k=0$ for all $k\in\mathbb{Z}$
(that is, we have noninformative observations).
Clearly the tail $\sigma$-field $\bigcap_n\mathcal{X}_{-\infty,n}$
is nontrivial, and thus the filter fails to be ergodic by Theorem
\ref{thm:kunita}.  Nonetheless, we claim that the conditional $K$-property
holds, so that our main
results apply for any equimeasurable loss; in particular, the filter
is pathwise optimal.

The key point is that, even in the hidden Markov model setting, one need 
not choose the ``canonical'' generating $\sigma$-field 
$\mathcal{X}=\mathcal{X}_{-\infty,0}$ in Definition \ref{def:ckauto}. In 
the present example, we choose instead 
$\mathcal{X}=\sigma\{\mathbf{1}_{X_k>1/2}:k\le 0\}$. 
To verify the conditional $K$-property,
note that $(\mathbf{1}_{X_k>1/2})_{k\in\mathbb{Z}}$ are i.i.d.\
$\mathop{\mathrm{Bernoulli}}(1/2)$ random variables and
$$
	X_k = \sum_{\ell=0}^\infty 2^{-\ell-1}\mathbf{1}_{X_{k+\ell}>1/2}
	\quad\mbox{a.s.}\quad\mbox{for all }k\in\mathbb{Z}.
$$
Thus $\mathcal{X}\subset T^{-1}\mathcal{X}$ by construction,
$\bigvee_k T^{-k}\mathcal{X}=\sigma\{X_n:n\in\mathbb{Z}\}$ is a 
generating $\sigma$-field, and $\bigcap_k T^k\mathcal{X}$ is trivial 
by the Kolmogorov zero-one law.
\end{example}

Let us now consider the decision problem in the setting of a hidden
Markov model with equimeasurable loss function $\ell_0(u)=l(u,X_0)$.
If the filter is ergodic, then Corollary \ref{cor:equiisgreat} ensures
that the mean-optimal strategy $\mathbf{\tilde u}$ is pathwise optimal.
In this setting, the mean-optimal strategy can be expressed
in terms of the filter:
$$
	\tilde u_k = 
	\mathop{\mathrm{arg\,min}}_{u\in U}
	\mathbf{E}[l(u,X_k)|Y_0,\ldots,Y_k] =
	\mathop{\mathrm{arg\,min}}_{u\in U}
	\int l(u,x)\,\Pi_k(dx).
$$
When $X_k$ takes values in a finite set $E=\{1,\ldots,d\}$, the filter
can be recursively computed in a straightforward manner \cite{CMR05}.
In this case, the mean-optimal strategy $\mathbf{\tilde u}$ can be 
implemented directly.  On the other hand, when $E$ is a continuous space,
the conditional measure $\Pi_k$ is an infinite-dimensional
object which cannot be computed exactly except in special cases.  However,
$\Pi_k$ can often be approximated very efficiently by recursive Monte Carlo 
approximations $\Pi_k^N=\frac{1}{N}\sum_{i=1}^N\delta_{Z_k^N(i)}$, known 
as particle filters \cite{CMR05}, that converge to the true filter $\Pi_k$
as the number of particles increases $N\to\infty$.  This suggests to approximate 
the mean-optimal strategy $\mathbf{\tilde u}$ by
$$
	\tilde u_k \approx \tilde u_k^N :=
	\mathop{\mathrm{arg\,min}}_{u\in U}
	\int l(u,x)\,\Pi_k^N(dx) =
	\mathop{\mathrm{arg\,min}}_{u\in U}
	\frac{1}{N}\sum_{i=1}^N
	l(u,Z_k^N(i)).
$$
The strategy $\mathbf{\tilde u}^N$ is a type of sequential stochastic
programming algorithm to approximate the mean-optimal strategy.
In this setting, it is of interest to establish whether the strategy
$\mathbf{\tilde u}^N$ is in fact approximately pathwise optimal, at
least in the weak sense.
To this end, we prove the following approximation lemma.

\begin{lemma}
\label{lem:pfilt}
In the hidden Markov model setting with equimeasurable loss
$\ell_0(u)=l(u,X_0)$, suppose that the filter is ergodic, and 
let $\Pi_k^N$ be an approximation of $\Pi_k$.  If
$$
	\lim_{N\to\infty}
	\limsup_{T\to\infty}
	\mathbf{E}\Bigg[
	\frac{1}{T}\sum_{k=1}^T 
	\esssup_{u\in\mathbb{U}_{0,k}}|\Pi_k^N(l(u,\cdot\,))-
	\Pi_k(l(u,\cdot\,))|
	\Bigg] = 0,
$$
then the strategy $\mathbf{\tilde u}^N$ is approximately
weakly pathwise optimal in the sense that
$$
	\lim_{N\to\infty}
	\liminf_{T\to\infty}
	\mathbf{P}[L_T(\mathbf{u})-L_T(\mathbf{\tilde u}^N)\ge
	-\varepsilon]=1\quad
	\mbox{for every }\varepsilon>0
$$
holds for every admissible strategy $\mathbf{u}$.
\end{lemma}

\begin{proof}
We begin by noting that
$$
	\mathbf{P}[L_T(\mathbf{u})-L_T(\mathbf{\tilde u}^N)<
	-\varepsilon] \le
	\mathbf{P}[L_T(\mathbf{u})-L_T(\mathbf{\tilde u})<-\varepsilon/2] +
	\mathbf{P}[L_T(\mathbf{\tilde u}^N)-L_T(\mathbf{\tilde u})>
	\varepsilon/2].
$$
Under the present assumptions, the mean-optimal strategy
$\mathbf{\tilde u}$ is (weakly) pathwise optimal.  It follows\footnote{
	As particle filters employ a random sampling mechanism, the strategy
	$\mathbf{\tilde u}^N$ is technically speaking not admissible in the
	sense of this paper: $\Pi_k^N$ (and therefore $\tilde u_k^N$) depends
	also on auxiliary sampling variables $\xi_0,\ldots,\xi_k$ that are 
	independent of $Y_0,\ldots,Y_k$.
	However, it is easily seen that all our results still hold when 
	such \emph{randomized} strategies are considered.  Indeed, it
	suffices to condition on $(\xi_k)_{k\ge 0}$, so that all our 
	results apply immediately under the conditional distribution.
}
as in the proof of Lemma \ref{lem:wkimplmn} that
$\mathbf{E}[(L_T(\mathbf{\tilde u}^N)-L_T(\mathbf{\tilde u}))_-]\to 0$
as $T\to\infty$, 
and we obtain for any admissible strategy $\mathbf{u}$ and
$\varepsilon>0$
$$
	\limsup_{T\to\infty}
	\mathbf{P}[L_T(\mathbf{u})-L_T(\mathbf{\tilde u}^N)<
	-\varepsilon] \le
	\frac{2}{\varepsilon}
	\limsup_{T\to\infty}
	\mathbf{E}[L_T(\mathbf{\tilde u}^N)-L_T(\mathbf{\tilde u})].
$$
To proceed, we estimate
\begin{align*}
	\mathbf{E}[L_T(\mathbf{\tilde u}^N)-L_T(\mathbf{\tilde u})] &=
	\mathbf{E}\Bigg[
	\frac{1}{T}\sum_{k=1}^T\int
	\{l(\tilde u_k^N,x)-l(\tilde u_k,x)\}\,\Pi_k(dx)
	\Bigg] \\
	&\le
	\mathbf{E}\Bigg[
	\frac{1}{T}\sum_{k=1}^T\int
	\{l(\tilde u_k^N,x)-l(\tilde u_k,x)\}\,\Pi_k^N(dx)
	\Bigg] \\
	&\qquad  + 2\,
	\mathbf{E}\Bigg[
	\frac{1}{T}\sum_{k=1}^T
	\esssup_{u\in\mathbb{U}_{0,k}}
	|\Pi_k^N(l(u,\cdot\,))-\Pi_k(l(u,\cdot\,))|
	\Bigg].
\end{align*}
But note that by the definition of $\mathbf{\tilde u}^N$
$$
	\int\{l(\tilde u_k^N,x)-l(\tilde u_k,x)\}\,\Pi_k^N(dx)
	=
	\inf_{u\in U}\int l(u,x)\,\Pi_k^N(dx) -
	\int l(\tilde u_k,x)\,\Pi_k^N(dx) \le 0.
$$
The proof is therefore easily completed by applying the assumption.
\qed\end{proof}

Evidently, the key difficulty in this problem is to control the 
time-average error of the filter approximation (in a norm determined by 
the loss function $l$) uniformly over the time horizon.  This problem is 
intimately related with the ergodic theory of nonlinear filters. The 
requisite property follows from the results in \cite{Han09spa} under 
reasonable ergodicity assumptions but under very stringent complexity 
assumptions on the loss (essentially that $\{l(u,\cdot\,):u\in U\}$ is 
uniformly Lipschitz).  Alternatively, one can apply the results in 
\cite{DL00}, which require exceedingly strong ergodicity assumptions but 
weaker complexity assumptions.  Let us note that one could similarly 
obtain a pathwise version of Lemma \ref{lem:pfilt}, but the requisite 
pathwise approximation property of particle filters has not been 
investigated in the literature.

\subsection{The conditions of Algoet, Weissman, Merhav, and Nobel}

The aim of this section is to briefly discuss the assumptions imposed in 
previous work on the pathwise optimality property due to Algoet 
\cite{Alg94}, Weissman and Merhav \cite{WM04}, and Nobel \cite{Nob03}. 
Let us emphasize that, while our results cover a much broader 
range of decision problems, none of these previous results follow in 
their entirety from our general results.  This highlights once more that 
our results are, unfortunately, nowhere close to a complete 
characterization of the pathwise optimality property.

\subsubsection{Algoet}

Algoet's results \cite{Alg94}, which cover the full information setting 
only, were already discussed at length in the introduction and in 
section \ref{sec:popt}.  The existence of a pathwise optimal strategy 
can be obtained in this setting under no additional assumptions from 
Theorem \ref{thm:sopt}, which even goes beyond Algoet's result in that 
it gives an explicit expression for the optimal asymptotic loss. 
However, Algoet establishes that in fact the mean-optimal strategy 
$\mathbf{\tilde u}$ is pathwise optimal in this setting, while our 
general Corollary \ref{cor:smeanopt} can only establish this under an 
additional complexity assumption.  We do not know whether this complexity
assumption can be weakened in general.

\subsubsection{Weissman and Merhav}

Weissman and Merhav \cite{WM04} consider the stochastic process setting 
$(X,Y)$, where $X_k$ takes values in $\{0,1\}$ and $Y_k$ takes values in 
$\mathbb{R}$ for all $k\in\mathbb{Z}$, and where the loss function takes 
the form $\ell_0(u)=l(u,X_1)$ and is assumed to be uniformly bounded.  
As $X$ is binary-valued, it is immediate that any loss function $l$ is 
equimeasurable.  Therefore, our results show that the mean-optimal 
strategy $\mathbf{\tilde u}$ is pathwise optimal whenever the model is a 
conditional $K$-automorphism relative to $\mathcal{Y}_{-\infty,0}$.

The assumption imposed by Weissman and Merhav in \cite{WM04} is as follows:
$$
	\sum_{k=1}^\infty
	\sup_{r\ge 1}\mathbf{E}[
	|\mathbf{P}[X_{r+k}=a|X_r=a,\mathcal{Y}_{0,r+k-1}]-
	\mathbf{P}[X_{r+k}=a|\mathcal{Y}_{0,r+k-1}]
	|]<\infty
	\mbox{ for }a=0,1.
$$
Using stationarity, this condition is equivalent to 
$$
	\sum_{k=0}^\infty
	\sup_{r\ge 1}\mathbf{E}[
	|\mathbf{P}[X_{1}=a|X_{-k}=a,\mathcal{Y}_{-r-k,0}]-
	\mathbf{P}[X_{1}=a|\mathcal{Y}_{-r-k,0}]
	|]<\infty
	\mbox{ for }a=0,1,
$$
which readily implies
$$
	\sum_{k=0}^\infty
	\mathbf{E}[
	|\mathbf{P}[X_{1}=a|\sigma\{X_{-k}\}\vee\mathcal{Y}_{-\infty,0}]-
	\mathbf{P}[X_{1}=a|\mathcal{Y}_{-\infty,0}]
	|]<\infty.
$$
If the $\sigma$-field $\sigma\{X_{-k}\}\vee\mathcal{Y}_{-\infty,0}$ 
could be replaced by the larger $\sigma$-field 
$\mathcal{X}_{-\infty,-k}\vee\mathcal{Y}_{-\infty,0}$ in this 
expression, then Assumption 3 of Corollary \ref{cor:smeanopt}
would follow immediately.  However, the smaller $\sigma$-field 
appears to yield a slightly better variant of the assumption imposed
in \cite{WM04}.  This is possible because
the result is restricted to the special choice of loss $\ell_0(u)=
l(u,X_1)$ that depends on $X_1$ only.
On the other hand, it is to be expected that in most cases the 
assumption of \cite{WM04} is much more stringent than that of Corollary 
\ref{cor:smeanopt}.  Note that Assumption 3 of Corollary 
\ref{cor:smeanopt} is purely qualitative in nature: 
it states, roughly speaking, that two $\sigma$-fields coincide.  This is 
a structural property of the model.  On the other hand, the assumption 
of \cite{WM04} is inherently \emph{quantitative} in nature: it requires 
that a certain mixing property holds at a sufficiently fast rate (the 
mixing coefficients must be summable).  A quantitative bound on 
the mixing rate is both much more restrictive and much harder to verify, 
in general, as compared to a purely structural property.

In a sense, the approach of Weissman and Merhav is much closer in spirit 
to the weak pathwise optimality results in this paper than it is to the
pathwise optimality results.  Indeed, if 
we replace the weak pathwise optimality property
$$
	\mathbf{P}[L_T(\mathbf{u})-L_T(\mathbf{u}^\star)< -\varepsilon]
	\xrightarrow{T\to\infty}0\quad\mbox{for every }\varepsilon>0
$$
by its quantitative counterpart
$$
	\sum_{T=1}^\infty
	\mathbf{P}[L_T(\mathbf{u})-L_T(\mathbf{u}^\star)< -\varepsilon]
	<\infty\quad\mbox{for every }\varepsilon>0,
$$
then pathwise optimality will automatically follow from the 
Borel-Cantelli lemma.
In the same spirit, if in Theorem \ref{thm:wkopt} we replace the
uniform conditional mixing assumption by the corresponding
quantitative counterpart
$$
	\sum_{k=1}^T
	\mathbf{E}\bigg[
	\esssup_{u,u'\in\mathbb{U}_{0}}
	|\mathbf{E}[\{\bar\ell_0^M(u)\circ T^{-k}\} ~\bar\ell_0^M(u')|
	\mathcal{Y}_{-\infty,0}]|
	\bigg]
	= O(T^\alpha)
$$
for some $\alpha<1$ (that may depend on $M$), then we easily obtain a pathwise
version of Lemma \ref{lem:wkerg} below (using Etemadi's well-known device
\cite{Ete81}), and consequently the conclusion of Theorem \ref{thm:wkopt}
is replaced by that of Theorem \ref{thm:sopt}.  It is unclear whether 
such quantitative mixing conditions provide a distinct mechanism for
pathwise optimality as compared to qualitative structural conditions as
in our main results.

\subsubsection{Nobel}

Nobel \cite{Nob03} considers the stochastic process setting $(X,Y)$ with 
observations of the additive form $Y_k=X_k+N_k$, where 
$N=(N_k)_{k\in\mathbb{Z}}$ is an $L^2$-martingale difference 
sequence independent of $X$.  The loss function considered is the mean-square 
loss $\ell_0(u) = (u-X_1)^2$.  This very special scenario is essential 
for the result given in \cite{Nob03}; on the other hand, it is not 
assumed that $(X,Y)$ is even stationary or that the decision space $U$ is 
a compact set (when $U=\mathbb{R}$, the quadratic loss is not 
dominated). In order to compare with our general results, we will 
additionally assume that $(X,Y)$ is stationary and ergodic and that 
$X_k$ are uniformly bounded random variables (so that we may choose 
$U=[-\|X_1\|_\infty,\|X_1\|_\infty]$ without loss of generality).

While this is certainly a decision problem with partial information,
the key observation is that this special problem is in fact a decision
problem with full information in disguise.  Indeed, note that we
can write for any strategy $\mathbf{u}$
$$
	L_T(\mathbf{u}) =
	\frac{1}{T}\sum_{k=1}^T (u_k-Y_{k+1})^2 +
	\frac{1}{T}\sum_{k=1}^T \{X_{k+1}^2 - Y_{k+1}^2\} +
	\frac{1}{T}\sum_{k=1}^T 2u_kN_{k+1}.
$$
The last term of this expression converges to zero a.s.\ as $T\to\infty$
for any admissible strategy $\mathbf{u}$ by the martingale law of large
numbers, as $(u_kN_{k+1})_{k\in\mathbb{Z}}$ 
is an $L^2$-martingale difference sequence.  On the other hand, the
second to last term of this expression does not depend on the strategy
$\mathbf{u}$ at all.  Therefore,
$$
	\liminf_{T\to\infty}\{L_T(\mathbf{u})-L_T(\mathbf{\tilde u})\} =
	\liminf_{T\to\infty}\Bigg\{
	\frac{1}{T}\sum_{k=1}^T (u_k-Y_{k+1})^2 -
	\frac{1}{T}\sum_{k=1}^T (\tilde u_k-Y_{k+1})^2 \Bigg\}
	\quad\mbox{a.s.},
$$
which corresponds to the decision problem with the full information 
loss $\ell_0(u)=(u-Y_1)^2$.  Thus pathwise optimality of the mean-optimal
strategy $\mathbf{\tilde u}$ follows from Algoet's result.  (The main
difficulty in \cite{Nob03} is to introduce suitable truncations to deal
with the lack of boundedness, which we avoided here.)

Of course, we could deduce the result from our general theory in the 
same manner: reduce first to a full information decision problem as 
above, and then invoke Corollary \ref{cor:smeanopt} in the full 
information setting.  However, a more relevant test of our general 
theory might be to ask whether one can deduce the result directly from 
Corollary \ref{cor:smeanopt}, without first reducing to the full 
information setting.  Unfortunately, it is not clear whether it is 
possible, in general, to find a generating $\sigma$-field $\mathcal{X}$ 
such that Assumption 3 of Corollary \ref{cor:smeanopt} holds.

One might interpret the additive noise model as a type of 
``informative'' observations: while $X$ cannot be reconstructed from the 
observations $Y$, the law of $X$ can certainly be reconstructed from the 
law of $Y$ even if the former were not known \emph{a priori} (this idea 
is exploited in \cite{WM04, Nob03} to devise universal prediction 
strategies that do not require prior knowledge of the law of $X$).  In 
the hidden Markov model setting, there is in fact a connection between 
``informative'' observations and the conditional $K$-property.  In 
particular, if $(X,Y)$ is a hidden Markov model where $X_k$ takes a 
finite number of values, and $Y_k=X_k+\xi_k$ where $\xi_k$ are i.i.d.\ 
and independent of $X$, then the conditional $K$-property holds, and we 
therefore have pathwise optimal strategies for \emph{any} dominated 
loss. This follows from observability conditions in the Markov setting, 
cf.\ \cite[section 6.2]{CvH10} and the references therein. However, the 
ideas that lead to this result do not appear to extend to more general 
situations.

\section{Proofs}
\label{sec:proofs}

\subsection{Proof of Theorem \ref{thm:sopt}}
\label{sec:proofsopt}

Throughout the proof, we fix a generating $\sigma$-field
$\mathcal{X}$ that satisfies the conditions of Theorem \ref{thm:sopt}.
In the following, we define the $\sigma$-fields
$$
	\mathcal{G}_k^n = 
	\mathcal{Y}_{-\infty,k}
	\vee T^{n-k}\mathcal{X},\qquad\quad
	\mathcal{G}_k^\infty = \bigcap_n \mathcal{G}_k^n.
$$
Note that $\mathcal{G}_k^n$ is decreasing in $n$ and increasing in $k$.

We begin by establishing the following lemma.

\begin{lemma}
\label{lem:mtgdiff}
For any admissible strategy $\mathbf{u}$ and any $m,n\in\mathbb{Z}$
$$
	\frac{1}{T}\sum_{k=1}^T\{
	\mathbf{E}[\ell_k(u_k)|\mathcal{G}_k^m]
	-\mathbf{E}[\ell_k(u_k)|
	\mathcal{G}_k^n]\}\xrightarrow{T\to\infty}0
	\quad\mbox{a.s.}
$$
\end{lemma}

\begin{proof}
Assume $m<n$ without loss of generality.  Fix $r<\infty$, and define
$$
	\Delta_k^j = 
	\mathbf{E}[\ell_k(u_k)\mathbf{1}_{\Lambda\circ T^k\le r}
	|\mathcal{G}_k^j] -
	\mathbf{E}[\ell_k(u_k)\mathbf{1}_{\Lambda\circ T^k\le r}
	|\mathcal{G}_k^{j+1}]
$$
for $m\le j<n$.
Then it is easily seen that we have the inequality
\begin{multline*}
	\bigg|\frac{1}{T}\sum_{k=1}^T\{
	\mathbf{E}[\ell_k(u_k)|\mathcal{G}_k^m]
	-\mathbf{E}[\ell_k(u_k)|
	\mathcal{G}_k^n]\}\bigg| \le
	\sum_{j=m}^{n-1}
	\bigg|\frac{1}{T}\sum_{k=1}^T
	\Delta_k^j\bigg| + \mbox{}\\
	\frac{1}{T}\sum_{k=1}^T \{
	\mathbf{E}[\Lambda\mathbf{1}_{\Lambda>r}|\mathcal{G}_0^m]+
	\mathbf{E}[\Lambda\mathbf{1}_{\Lambda>r}|\mathcal{G}_0^n]
	\}\circ T^k.
\end{multline*}
By the ergodic theorem, the second term on the right converges
to $\kappa(r):=\mathbf{E}[2\Lambda\mathbf{1}_{\Lambda>r}]$
a.s.\ as $T\to\infty$.  It remains to consider the first term.

To this end, note the inclusions $\mathcal{G}_k^{j+1}\subseteq 
\mathcal{G}_k^j\subseteq\mathcal{G}_{k+1}^{j+1}$.
It follows that 
$$
	\Delta_k^j \mbox{ is }
	\mathcal{G}_{k+1}^{j+1}\mbox{-measurable},\quad
	\mathbf{E}[\Delta_k^j|\mathcal{G}_k^{j+1}]=0,\quad
	\mbox{and }
	|\Delta_k^j|\le 2r
$$
for $0\le j<n$.
Thus $(\Delta_k^j)_{k\ge 1}$ is a uniformly bounded martingale difference 
sequence with respect to the filtration $(\mathcal{G}^{j+1}_{k+1})_{k\ge 
1}$, and we consequently have
$$
	\frac{1}{T}\sum_{k=1}^T\Delta_k^j\xrightarrow{T\to\infty}0
	\quad\mbox{a.s.}
$$
by the simplest form of the martingale law of large numbers (indeed,
it is easily seen that $M_n = \sum_{k=1}^n \Delta_k^j/k$ is an 
$L^2$-bounded martingale, so that the result follows from the martingale 
convergence theorem and Kronecker's lemma).

Putting together these results, we obtain
$$
	\limsup_{T\to\infty}
	\bigg|\frac{1}{T}\sum_{k=1}^T\{
	\mathbf{E}[\ell_k(u_k)|\mathcal{G}_k^m]
	-\mathbf{E}[\ell_k(u_k)|
	\mathcal{G}_k^n]\}\bigg| \le
	\kappa(r)\quad\mbox{a.s.}
$$
for arbitrary $r<\infty$. Letting $r\to\infty$ completes the proof.
\qed\end{proof}

We can now establish a lower bound on the loss of any strategy.

\begin{corollary}
\label{cor:mtgdiff}
Under the assumptions of Theorem \ref{thm:sopt}, we have
$$
	\frac{1}{T}\sum_{k=1}^T\{
	\ell_k(u_k)-\mathbf{E}[\ell_k(u_k)|
	\mathcal{G}^\infty_k]\}\xrightarrow{T\to\infty}0
	\quad\mbox{a.s.}
$$
for any admissible strategy $\mathbf{u}$.  In particular,
$$
	\liminf_{T\to\infty} L_T(\mathbf{u}) \ge
	\mathbf{E}\Big[
	\essinf_{u\in\mathbb{U}_0}
	\mathbf{E}[\ell_0(u)|\mathcal{G}_0^\infty]
	\Big] = L^\star\quad
	\mbox{a.s.}
$$
\end{corollary}

\begin{proof}
We begin by noting that
\begin{align*}
	&\bigg|
	\frac{1}{T}\sum_{k=1}^T\{
	\mathbf{E}[\ell_k(u_k)|
	\mathcal{G}_k^n]-
	\mathbf{E}[\ell_k(u_k)|
	\mathcal{G}_k^\infty]\}
	\bigg| \le
	\frac{1}{T}\sum_{k=1}^T
	\esssup_{u\in\mathbb{U}_k}
	|\mathbf{E}[\ell_k(u)|
	\mathcal{G}_k^n]-
	\mathbf{E}[\ell_k(u)|
	\mathcal{G}_k^\infty]|
	\\ &\qquad\qquad\mbox{}
	\xrightarrow{T\to\infty}
	\mathbf{E}\bigg[
	\esssup_{u\in\mathbb{U}_0}
	|\mathbf{E}[\ell_0(u)|
	\mathcal{G}_0^n]-
	\mathbf{E}[\ell_0(u)|
	\mathcal{G}_0^\infty]|
	\bigg]\quad\mbox{a.s.}
\end{align*}
by the ergodic theorem.  Similarly,
\begin{align*}
	&\bigg|
	\frac{1}{T}\sum_{k=1}^T\{
	\mathbf{E}[\ell_k(u_k)|
	\mathcal{G}_k^m]-\ell_k(u_k)\}
	\bigg| \le
	\frac{1}{T}\sum_{k=1}^T
	\esssup_{u\in\mathbb{U}_k}
	|\mathbf{E}[\ell_k(u)|
	\mathcal{G}_k^m]-\ell_k(u)|
	\\ &\qquad\qquad\mbox{}
	\xrightarrow{T\to\infty}
	\mathbf{E}\bigg[
	\esssup_{u\in\mathbb{U}_0}
	|\mathbf{E}[\ell_0(u)|
	\mathcal{G}_0^m]-\ell_0(u)|
	\bigg]\quad\mbox{a.s.}
\end{align*}
Therefore, using Lemma \ref{lem:mtgdiff} and Assumption 2 of Theorem
\ref{thm:sopt}, the first statement of the Corollary follows by letting 
$n\to\infty$ and $m\to-\infty$.

For the second statement, it suffices to note that
$$
	\frac{1}{T}\sum_{k=1}^T
	\mathbf{E}[\ell_k(u_k)|\mathcal{G}^\infty_k]
	\ge
	\frac{1}{T}\sum_{k=1}^T
	\essinf_{u\in\mathbb{U}_k}
	\mathbf{E}[\ell_k(u)|\mathcal{G}^\infty_k]
	\xrightarrow{T\to\infty} L^\star\quad\mbox{a.s.}
$$
by the ergodic theorem and Assumption 3 of Theorem \ref{thm:sopt}.
\qed\end{proof}

As was explained in the introduction, a pathwise optimal strategy
could easily be obtained of one can prove ``ergodic tower property''
of the form 
$$
	\frac{1}{T}\sum_{k=1}^T\{\ell_k(u_k)-\mathbf{E}[\ell_k(u_k)|
	\mathcal{Y}_{0,k}]\}\xrightarrow{T\to\infty}0
	\quad\mbox{a.s.}\quad?
$$
Corollary \ref{cor:mtgdiff} establishes just such a property, but where the 
$\sigma$-field $\mathcal{Y}_{0,k}$ is replaced by the larger 
$\sigma$-field $\mathcal{G}^\infty_k$.  This yields a lower 
bound on the asymptotic loss, but it is far
from clear that one can choose a $\mathcal{Y}_{0,k}$-adapted strategy
that attains this bound.

Therefore, what remains is to show that there exists an admissible 
strategy $\mathbf{u}^\star$ that attains the lower bound in Corollary 
\ref{cor:mtgdiff}.  A promising candidate is the mean-optimal strategy 
$\mathbf{\tilde u}$. Unfortunately, we are not able to prove  
pathwise optimality of the mean-optimal strategy in the general setting 
of Theorem \ref{thm:sopt}.  However, we will obtain a pathwise optimal 
strategy $\mathbf{u}^\star$ by a judicious modification of the 
mean-optimal strategy $\mathbf{\tilde u}$. The key idea is the following 
``uniform'' version of the martingale convergence theorem, which we 
prove following Neveu \cite[Lemma V-2-9]{Nev75}.

\begin{lemma}
\label{lem:neveu}
The following holds:
$$
	\essinf_{u\in\mathbb{U}_{-k,0}}
	\mathbf{E}[\ell_0(u)|\mathcal{Y}_{-k,0}] 
	\xrightarrow{k\to\infty} 
	\essinf_{u\in\mathbb{U}_0}
	\mathbf{E}[\ell_0(u)|\mathcal{Y}_{-\infty,0}]
	\quad\mbox{a.s.\ and in }L^1.
$$
\end{lemma}

\begin{proof}
Using the construction of the essential supremum
as in the proof of Lemma \ref{lem:meanopt}, we can choose 
for each $0\le k<\infty$ a countable family 
$\mathbb{U}_{-k,0}^{\rm c}\subset\mathbb{U}_{-k,0}$ such that
$$
	\essinf_{u\in\mathbb{U}_{-k,0}}
	\mathbf{E}[\ell_0(u)|\mathcal{Y}_{-k,0}] =
	\inf_{u\in\mathbb{U}_{-k,0}^{\rm c}}
	\mathbf{E}[\ell_0(u)|\mathcal{Y}_{-k,0}]
	\quad\mbox{a.s.},
$$
and a countable family $\mathbb{U}_0^{\rm c}\subset\mathbb{U}_0$ such that
$$
	\essinf_{u\in\mathbb{U}_0}
	\mathbf{E}[\ell_0(u)|\mathcal{Y}_{-\infty,0}] =
	\inf_{u\in\mathbb{U}_0^{\rm c}}
	\mathbf{E}[\ell_0(u)|\mathcal{Y}_{-\infty,0}]
	\quad\mbox{a.s.}
$$
For every $0\le k<\infty$, choose an arbitrary ordering 
$(U_k^n)_{n\in\mathbb{N}}$ of the elements of the countable set
$\mathbb{U}_{-k,0}^{\rm c}\cup(\mathbb{U}_0^{\rm c}\cap
\mathbb{U}_{-k,0})$.  Then we clearly have
$$
	M_k := \essinf_{u\in\mathbb{U}_{-k,0}}
	\mathbf{E}[\ell_0(u)|\mathcal{Y}_{-k,0}] =
	\min_{0\le l\le k}\inf_{n\in\mathbb{N}}
	\mathbf{E}[\ell_0(U_l^n)|\mathcal{Y}_{-k,0}]
	\quad\mbox{a.s.}
$$
and
$$
	M := \essinf_{u\in\mathbb{U}_0}
	\mathbf{E}[\ell_0(u)|\mathcal{Y}_{-\infty,0}] =
	\inf_{0\le l<\infty}\inf_{n\in\mathbb{N}}
	\mathbf{E}[\ell_0(U_l^n)|\mathcal{Y}_{-\infty,0}]
	\quad\mbox{a.s.}
$$
Our aim is to prove that $M_k\to M$ a.s.\ and in $L^1$ as $k\to\infty$.

We begin by noting that $|M_k|\le \mathbf{E}[\Lambda|
\mathcal{Y}_{-k,0}]$.  Therefore, the sequence $(M_k)_{k\ge 0}$ is 
uniformly integrable.  Moreover, $(M_k)_{k\ge 0}$ is a supermartingale 
with respect to the filtration $(\mathcal{Y}_{-k,0})_{k\ge 0}$: indeed, we 
can easily compute
$$
	\mathbf{E}[M_{k+1}|\mathcal{Y}_{-k,0}] \le
	\mathbf{E}\Big[
	\min_{0\le l\le k}\inf_{n\in\mathbb{N}}
	\mathbf{E}[\ell_0(U_l^n)|\mathcal{Y}_{-k-1,0}]
	\Big|\mathcal{Y}_{-k,0}\Big]
	\le M_k.
$$
Thus $M_k\to M_\infty$ a.s.\ and in $L^1$ by the martingale 
convergence theorem for some random variable $M_\infty$.  We must now show 
that $M_\infty=M$ a.s.  Note that
$$
	M_\infty = \lim_{k\to\infty} M_k \le 
	\lim_{k\to\infty}\mathbf{E}[\ell_0(U_l^n)|\mathcal{Y}_{-k,0}]
	=
	\mathbf{E}[\ell_0(U_l^n)|\mathcal{Y}_{-\infty,0}]\quad
	\mbox{a.s.}
$$
for every $n\in\mathbb{N}$ and $0\le l<\infty$, so $M_\infty\le M$ a.s.
To complete the proof, it therefore suffices to show that 
$\mathbf{E}[M_\infty]=\mathbf{E}[M]$.

To this end, define for $N\in\mathbb{N}$ and $0\le k\le\infty$
$$
	M_k^N = 
	\min_{l\le N\wedge k}\min_{n\le N}
	\mathbf{E}[\ell_0(U_l^n)|\mathcal{Y}_{-k,0}].
$$
As $(M_k^N)_{k\ge 0}$ is again a supermartingale, clearly
$\mathbf{E}[M_k^N]$ is doubly nonincreasing in $k$ and $N$.
The exchange of limits is therefore permitted, so that
$$
	\mathbf{E}[M_\infty]
	=\lim_{k\to\infty}\lim_{N\to\infty}\mathbf{E}[M_k^N]
	=\lim_{N\to\infty}\lim_{k\to\infty}\mathbf{E}[M_k^N]
	=\mathbf{E}[M].
$$
This completes the proof.
\qed\end{proof}

\begin{corollary}
\label{cor:neveu}
Suppose that Assumption 3 of Theorem \ref{thm:sopt} holds.  Then
$$
	\mathbf{E}[\ell_k(\tilde u_k)|
        \mathcal{G}^\infty_k]\circ T^{-k}
	\xrightarrow{k\to\infty}
	\essinf_{u\in\mathbb{U}_0}
	\mathbf{E}[\ell_0(u)|\mathcal{G}^\infty_0]
	\quad\mbox{in }L^1.	
$$
\end{corollary}

\begin{proof}
Define $\hat u_k = \tilde u_k\circ T^{-k} \in \mathbb{U}_{-k,0}$, 
so that
$$
	\mathbf{E}[\ell_k(\tilde u_k)|
        \mathcal{G}^\infty_k]\circ T^{-k} =
	\mathbf{E}[\ell_0(\hat u_k)|
        \mathcal{G}^\infty_0].
$$
By stationarity and the definition of $\mathbf{\tilde u}$, we have
\begin{align*}
	\mathbf{E}[
	\mathbf{E}[\ell_0(\hat u_k)|
        \mathcal{G}^\infty_0]] &=
	\mathbf{E}[
	\mathbf{E}[\ell_0(\hat u_k)|
        \mathcal{Y}_{-k,0}]] \\ &\mbox{}\le
	\mathbf{E}\Big[\essinf_{u\in\mathbb{U}_{-k,0}}
	\mathbf{E}[\ell_0(u)|
        \mathcal{Y}_{-k,0}]\Big] + k^{-1}.
\end{align*}
Therefore, by Lemma \ref{lem:neveu}, we have
$$
	\limsup_{k\to\infty}
	\mathbf{E}[
	\mathbf{E}[\ell_0(\hat u_k)|
        \mathcal{G}^\infty_0]] \le
	\mathbf{E}\Big[\essinf_{u\in\mathbb{U}_0}
	\mathbf{E}[\ell_0(u)|
        \mathcal{Y}_{-\infty,0}]\Big] = L^\star.
$$
On the other hand, note that
$$
	\mathbf{E}[\ell_0(\hat u_k)|
        \mathcal{G}^\infty_0] \ge
	\essinf_{u\in\mathbb{U}_0}
	\mathbf{E}[\ell_0(u)|
        \mathcal{G}^\infty_0]\quad\mbox{a.s.}
$$
Using Assumption 3, we therefore have
$$
	\limsup_{k\to\infty}
	\Big\|	\mathbf{E}[\ell_0(\hat u_k)|
        \mathcal{G}^\infty_0] -
	\essinf_{u\in\mathbb{U}_0}
	\mathbf{E}[\ell_0(u)|
        \mathcal{G}^\infty_0]\Big\|_1 
	\le 0.
$$
This completes the proof.
\qed\end{proof}

We are now in the position to construct the pathwise optimal strategy
$\mathbf{u}^\star$.  By Corollary \ref{cor:neveu}, we can choose a
(nonrandom) sequence $k_n\uparrow\infty$ such that
$$
	\mathbf{E}[\ell_{k_n}(\tilde u_{k_n})|
        \mathcal{G}^\infty_{k_n}]\circ T^{-k_n}
	\xrightarrow{n\to\infty}
	\essinf_{u\in\mathbb{U}_0}
	\mathbf{E}[\ell_0(u)|\mathcal{G}^\infty_0]
	\quad\mbox{a.s.}
$$
Let us define
$$
	u^\star_k = \tilde u_{k_n}\circ T^{k-k_n}\quad\mbox{for }
	k_n\le k<k_{n+1},~n\in\mathbb{N}.
$$
Then clearly $\mathbf{u}^\star=(u_k^\star)_{k\ge 1}$ is an
admissible strategy.

\begin{lemma}
\label{lem:optcost}
Suppose that the assumptions of Theorem \ref{thm:sopt} hold.  Then
$$
	\lim_{T\to\infty}L_T(\mathbf{u}^\star)= L^\star
	\quad\mbox{a.s.}
$$
\end{lemma}

\begin{proof}
By construction,
$$
	\mathbf{E}[\ell_k(u_k^\star)|
        \mathcal{G}^\infty_{k}]\circ T^{-k}
	\xrightarrow{k\to\infty}
	\essinf_{u\in\mathbb{U}_0}
	\mathbf{E}[\ell_0(u)|\mathcal{G}^\infty_0]
	\quad\mbox{a.s.}
$$
Moreover,
$$
	\sup_{k\ge 1}\big|\mathbf{E}[\ell_k(u_k^\star)|
        \mathcal{G}^\infty_{k}]\circ T^{-k}\big| \le
	\mathbf{E}[\Lambda|\mathcal{G}^\infty_0]\in L^1.
$$
Therefore, by Maker's generalized ergodic theorem \cite[Corollary 10.8]{Kal02}
$$
	\frac{1}{T}\sum_{k=1}^T
	\mathbf{E}[\ell_k(u_k^\star)|
        \mathcal{G}^\infty_{k}]
	\xrightarrow{T\to\infty}
	\mathbf{E}\Big[\essinf_{u\in\mathbb{U}_0}
	\mathbf{E}[\ell_0(u)|\mathcal{G}^\infty_0]\Big]=
	L^\star\quad\mbox{a.s.}
$$
Thus $L_T(\mathbf{u}^\star)\to L^\star$ a.s.\ as $T\to\infty$ by
Corollary \ref{cor:mtgdiff}.
\qed\end{proof}

The proof of Theorem \ref{thm:sopt} is now complete.
Indeed, if $\mathbf{u}$ is admissible, then
$$
	\liminf_{T\to\infty}\{L_T(\mathbf{u})-L_T(\mathbf{u}^\star)\}
	=\liminf_{T\to\infty}L_T(\mathbf{u})-L^\star\ge 0\quad
	\mbox{a.s.}
$$
by Lemma \ref{lem:optcost} and Corollary \ref{cor:mtgdiff}, so
$\mathbf{u}^\star$ is pathwise optimal.

\subsection{Proof of Corollary \ref{cor:smeanopt}}
\label{sec:proofsmeanopt}

The prove pathwise optimality, it suffices to show
$L_T(\mathbf{\tilde u})\to L^\star$ a.s.

\begin{lemma}
Under the assumptions of Corollary \ref{cor:smeanopt}, the mean-optimal
strategy $\mathbf{\tilde u}$ (Lemma \ref{lem:meanopt}) satisfies
$L_T(\mathbf{\tilde u})\to L^\star$ a.s.\ as $T\to\infty$.
\end{lemma}

\begin{proof}
By the definition of $\mathbf{\tilde u}$ and Lemma \ref{lem:neveu},
$$
	\mathbf{E}[\ell_k(\tilde u_k)|\mathcal{Y}_{0,k}]\circ T^{-k}
	\xrightarrow{k\to\infty}
	\essinf_{u\in\mathbb{U}_0}\mathbf{E}[\ell_0(u)|\mathcal{Y}_{-\infty,0}]
	\quad\mbox{a.s.}
$$
Therefore, the third part of Assumption 2 of Corollary \ref{cor:smeanopt} 
implies that
$$
	\mathbf{E}[\ell_k(\tilde u_k)|\mathcal{Y}_{-\infty,k}]\circ T^{-k}
	\xrightarrow{k\to\infty}
	\essinf_{u\in\mathbb{U}_0}\mathbf{E}[\ell_0(u)|\mathcal{Y}_{-\infty,0}]
	\quad\mbox{a.s.}
$$
But by Assumption 3 of Corollary \ref{cor:smeanopt} and stationarity, we obtain
$$
	\mathbf{E}[\ell_k(\tilde u_k)|\mathcal{G}_k^\infty]\circ T^{-k}
	\xrightarrow{k\to\infty}
	\essinf_{u\in\mathbb{U}_0}\mathbf{E}[\ell_0(u)|\mathcal{Y}_{-\infty,0}]
	\quad\mbox{a.s.}
$$
Moreover, we have
$$
	\sup_{k\ge 1}\big|\mathbf{E}[\ell_k(\tilde u_k)|
        \mathcal{G}^\infty_{k}]\circ T^{-k}\big| \le
	\mathbf{E}[\Lambda|\mathcal{G}^\infty_0]\in L^1.
$$
Maker's generalized ergodic theorem \cite[Corollary 10.8]{Kal02} 
therefore yields
$$
	\frac{1}{T}\sum_{k=1}^T
	\mathbf{E}[\ell_k(\tilde u_k)|
        \mathcal{G}^\infty_{k}]
	\xrightarrow{T\to\infty}
	\mathbf{E}\Big[\essinf_{u\in\mathbb{U}_0}
	\mathbf{E}[\ell_0(u)|\mathcal{Y}_{-\infty,0}]\Big]=
	L^\star\quad\mbox{a.s.}
$$
As the assumptions of Corollary \ref{cor:smeanopt} imply those of
Theorem \ref{thm:sopt}, the result as well as pathwise optimality
of $\mathbf{\tilde u}$ now follow from Corollary \ref{cor:mtgdiff}.
\qed\end{proof}

\subsection{Proof of Theorem \ref{thm:wkopt}}
\label{sec:proofwkopt}

The proof of the Theorem is once again based on a variant of the
``ergodic tower property'' described in the introduction.  
In the present setting, the result follows rather easily from
the conditional weak mixing assumption.

\begin{lemma}
\label{lem:wkerg}
Suppose that the assumption of Theorem \ref{thm:wkopt} holds.  Then
$$
	\frac{1}{T}\sum_{k=1}^T
	\{\ell_k(u_k)-\mathbf{E}[\ell_k(u_k)|\mathcal{Y}_{-\infty,k}]\}
	\xrightarrow{T\to\infty}0\quad\mbox{in }L^1
$$
for every admissible strategy $\mathbf{u}$.
\end{lemma}

\begin{proof}
Define $\bar\ell_k^M(u)=\bar\ell_0^M(u)\circ T^k$ for $u\in U$.
We begin by noting that
$$
	\mathbf{E}\bigg[\bigg(
	\frac{1}{T}\sum_{k=1}^T \bar\ell_k^M(u_k)
	\bigg)^2
	\bigg]=
	\frac{1}{T^2}\sum_{n,m=1}^T
	\mathbf{E}[
	\bar\ell_n^M(u_n) \bar\ell_m^M(u_m)
	].
$$
Suppose that $m\le n$.  Then by stationarity and as $\mathbf{u}$ is admissible
\begin{align*}
	\mathbf{E}[\bar\ell_n^M(u_n) \bar\ell_m^M(u_m)] 
	&=
	\mathbf{E}[\bar\ell_0^M(u_n\circ T^{-n}) ~
	\{\bar\ell_0^M(u_m\circ T^{-m})\circ T^{-(n-m)}\}] \\
	&\le
	\mathbf{E}\bigg[
	\esssup_{u,u'\in\mathbb{U}_0}
	|\mathbf{E}[\bar\ell_0^M(u') ~
	\{\bar\ell_0^M(u)\circ T^{-(n-m)}\}|
	\mathcal{Y}_{-\infty,0}]|\bigg].
\end{align*}
We can therefore estimate
\begin{align*}
	\mathbf{E}\bigg[\bigg(
	\frac{1}{T}\sum_{k=1}^T \bar\ell_k^M(u_k)
	\bigg)^2
	\bigg] &\le
	\frac{2}{T^2}\sum_{n=1}^T\sum_{k=0}^{n-1}
	\mathbf{E}\bigg[
	\esssup_{u,u'\in\mathbb{U}_0}|
	\mathbf{E}[\bar\ell_0^M(u') ~
	\{\bar\ell_0^M(u)\circ T^{-k}\}|
	\mathcal{Y}_{-\infty,0}]|\bigg] \\
	&=
	\frac{2}{T^2}\sum_{k=0}^{T-1}
	(T-k)\,
	\mathbf{E}\bigg[
	\esssup_{u,u'\in\mathbb{U}_0}|
	\mathbf{E}[\bar\ell_0^M(u') ~
	\{\bar\ell_0^M(u)\circ T^{-k}\}|
	\mathcal{Y}_{-\infty,0}]|\bigg] \\
	&\le
	\frac{2}{T}\sum_{k=0}^{T-1}
	\mathbf{E}\bigg[
	\esssup_{u,u'\in\mathbb{U}_0}|
	\mathbf{E}[\bar\ell_0^M(u') ~
	\{\bar\ell_0^M(u)\circ T^{-k}\}|
	\mathcal{Y}_{-\infty,0}]|\bigg].
\end{align*}
By the uniform conditional mixing assumption, it follows that
$$
	\lim_{M\to\infty}
	\limsup_{T\to\infty}
	\bigg\|
	\frac{1}{T}\sum_{k=1}^T \bar\ell_k^M(u_k)
	\bigg\|_2=0.
$$
On the other hand, note that
$$
	\sup_{T\ge 1}
	\bigg\|
	\frac{1}{T}\sum_{k=1}^T
	\{\ell_k(u_k)-\mathbf{E}[\ell_k(u_k)|\mathcal{Y}_{-\infty,k}]\}
	-
	\frac{1}{T}\sum_{k=1}^T \bar\ell_k^M(u_k)
	\bigg\|_1
	\le
	\mathbf{E}[2\Lambda\mathbf{1}_{\Lambda>M}]
	\xrightarrow{M\to\infty}0.
$$
The result now follows by applying the triangle inequality.
\qed\end{proof}

\begin{corollary}
\label{cor:wkerg}
Under the assumption of Theorem \ref{thm:wkopt}, we have
$$
	\mathbf{P}\Big[L_T(\mathbf{u})-L^\star\le -\varepsilon\Big]
	\xrightarrow{T\to\infty}0\quad\mbox{for every }\varepsilon>0
$$
for every admissible strategy $\mathbf{u}$.
\end{corollary}

\begin{proof}
Let $\mathbf{u}$ be any admissible strategy. Then by Lemma \ref{lem:wkerg}
$$
	L_T(\mathbf{u})-\frac{1}{T}\sum_{k=1}^T
	\mathbf{E}[\ell_k(u_k)|\mathcal{Y}_{-\infty,k}]
	\xrightarrow{T\to\infty}0\quad\mbox{in }L^1.
$$
On the other hand, note that
$$
	\frac{1}{T}\sum_{k=1}^T
        \mathbf{E}[\ell_k(u_k)|\mathcal{Y}_{-\infty,k}] \ge
	\frac{1}{T}\sum_{k=1}^T
	\essinf_{u\in\mathbb{U}_k}
        \mathbf{E}[\ell_k(u)|\mathcal{Y}_{-\infty,k}] 
	\xrightarrow{T\to\infty}L^\star\quad\mbox{in }L^1
$$
by the ergodic theorem.  The result follows directly.
\qed\end{proof}

In view of Corollary \ref{cor:wkerg}, in order to establish weak pathwise 
optimality of $\mathbf{\tilde u}$ it evidently suffices to prove that 
$\mathbf{\tilde u}$ satisfies the ergodic theorem $L_T(\mathbf{\tilde 
u}) \to L^\star$ in $L^1$.  However, most of the work was already done in
the proof of Theorem \ref{thm:sopt}.

\begin{lemma}
\label{lem:neveu0}
Under the assumption of Theorem \ref{thm:wkopt},
$L_T(\mathbf{\tilde u})\to L^\star$ in $L^1$.
\end{lemma}

\begin{proof}
By the definition of $\mathbf{\tilde u}$, we have
$$
	\mathbf{E}[\ell_k(\tilde u_k)|\mathcal{Y}_{0,k}]\circ T^{-k}
	\le
	\essinf_{u\in\mathbb{U}_{-k,0}}
	\mathbf{E}[\ell_0(u)|\mathcal{Y}_{-k,0}] + k^{-1}
	\quad\mbox{a.s.}
$$
Therefore, by Lemma \ref{lem:neveu}, we obtain
$$
	\limsup_{k\to\infty}\mathbf{E}[
	\mathbf{E}[\ell_k(\tilde u_k)|\mathcal{Y}_{-\infty,k}]\circ T^{-k}
	]
	\le 
	\mathbf{E}\bigg[
	\essinf_{u\in\mathbb{U}_0}
	\mathbf{E}[\ell_0(u)|\mathcal{Y}_{-\infty,0}]
	\bigg]=L^\star.
$$
On the other hand,
$$
	\mathbf{E}[\ell_k(\tilde u_k)|\mathcal{Y}_{-\infty,k}]\circ T^{-k}
	\ge
	\essinf_{u\in\mathbb{U}_0}
	\mathbf{E}[\ell_0(u)|\mathcal{Y}_{-\infty,0}]
	\quad\mbox{a.s.}
$$
for all $k\in\mathbb{N}$.  It follows that
\begin{multline*}
	\limsup_{k\to\infty}
	\bigg\|
	\mathbf{E}[\ell_k(\tilde u_k)|\mathcal{Y}_{-\infty,k}]\circ T^{-k}-
	\essinf_{u\in\mathbb{U}_0}
        \mathbf{E}[\ell_0(u)|\mathcal{Y}_{-\infty,0}]
	\bigg\|_1 = \\
	\limsup_{k\to\infty}\mathbf{E}[
	\mathbf{E}[\ell_k(\tilde u_k)|\mathcal{Y}_{-\infty,k}]\circ T^{-k}
	] - 
	\mathbf{E}\bigg[
	\essinf_{u\in\mathbb{U}_0}
	\mathbf{E}[\ell_0(u)|\mathcal{Y}_{-\infty,0}]
	\bigg] \le 0.
\end{multline*}
Therefore, by Maker's generalized ergodic theorem \cite[Corollary 10.8]{Kal02}
$$
	\frac{1}{T}\sum_{k=1}^T\mathbf{E}[\ell_k(\tilde u_k)|\mathcal{Y}_{-\infty,k}]
	\xrightarrow{T\to\infty}L^\star\quad\mbox{in }L^1.
$$
The result now follows using Lemma \ref{lem:wkerg}.
\qed\end{proof}

Combining Corollary \ref{cor:wkerg} and Lemma \ref{lem:neveu0} completes
the proof of Theorem \ref{thm:wkopt}.

\subsection{Proof of Theorem \ref{thm:wkmixcv}}
\label{sec:proofwkmixcv}

The implication $1\Rightarrow 2$ of Theorem \ref{thm:wkmixcv} follows 
immediately from Theorem \ref{thm:woptfinite}.  In the following, we 
will prove the converse implication $2\Rightarrow 1$: that is, we will 
show that if $(\Omega,\mathcal{B},\mathbf{P},T)$ is \emph{not} 
conditionally weak mixing relative to $\mathcal{Y}$, then we 
can construct a bounded loss function $\ell$ with some finite decision 
space $U$ for which there exists no weakly pathwise optimal strategy.

We begin by providing a ``diagonal'' characterization of conditional
weak mixing.

\begin{lemma}
\label{lem:diagcmix}
$(\Omega,\mathcal{B},\mathbf{P},T)$ is conditionally
weak mixing relative to $\mathcal{Z}$ if and only if
$$
	\frac{1}{T}\sum_{k=1}^T
	|\mathbf{E}[\{h\circ T^{-k}\}~h|\mathcal{Z}]
	- \mathbf{E}[h\circ T^{-k}|\mathcal{Z}]\,
	\mathbf{E}[h|\mathcal{Z}]
	|
	\xrightarrow{T\to\infty}0
	\quad\mbox{in }L^1
$$
for every $h\in L^2$, provided that $\mathcal{Z}\subseteq T^{-1}\mathcal{Z}$.
\end{lemma}

\begin{proof}
It suffices to show that if the equation display in the lemma holds, then
$(\Omega,\mathcal{B},\mathbf{P},T)$ is conditionally weak mixing relative
to $\mathcal{Z}$. To this end, let us fix $h\in L^2$ and denote by 
$\mathscr{A}$ the class of all functions $g\in L^2$ such that
$$
	\frac{1}{T}\sum_{k=1}^T
	|\mathbf{E}[\{g\circ T^{-k}\}~h|\mathcal{Z}]
	- \mathbf{E}[g\circ T^{-k}|\mathcal{Z}]\,
	\mathbf{E}[h|\mathcal{Z}]
	|
	\xrightarrow{T\to\infty}0
	\quad\mbox{in }L^1.
$$
Clearly $\mathscr{A}$ is closed linear subspace of $L^2$.
Note that $\mathscr{A}$ certainly contains every random variable
of the form $h\mathbf{1}_B\circ T^m$ or $\mathbf{1}_B\circ T^m$
for $m\in\mathbb{Z}$ and $B\in\mathcal{Z}$.  Therefore,
the closed linear span $K$ of all such random variables is included
in $\mathscr{A}$.  On the other hand, suppose that $g\in K^\perp$.
Then for every $k\in\mathbb{Z}$, we have
$$
	\mathbf{E}[
	\mathbf{E}[\{g\circ T^{-k}\}~h|\mathcal{Z}]\,\mathbf{1}_B]
 	= \mathbf{E}[
	g~ \{h\mathbf{1}_B\circ T^k\}] = 0
$$
for all $B\in\mathcal{Z}$.  It follows
that $\mathbf{E}[\{g\circ T^{-k}\}~h|\mathcal{Z}]=0$ a.s.\ 
for all $k\in\mathbb{N}$.  Similarly, we find that
$\mathbf{E}[g\circ T^{-k}|\mathcal{Z}]=0$ a.s.\ 
for all $k\in\mathbb{N}$.  Thus evidently $K^\perp\subseteq\mathscr{A}$
also.  Therefore, $\mathscr{A}$ contains $K\oplus K^\perp=L^2$, and
the proof is complete.
\qed\end{proof}

In the remainder of this section, we suppose that 
$(\Omega,\mathcal{B},\mathbf{P},T)$ is not 
conditionally weakly mixing relative to $\mathcal{Y}$.
By Lemma \ref{lem:diagcmix}, there is a function $h\in L^2$
such that
$$
	\limsup_{T\to\infty}
	\mathbf{E}\Bigg[
	\frac{1}{T}\sum_{k=1}^T
	|\mathbf{E}[\{H\circ T^{-k}\}~H|\mathcal{Y}]|
	\bigg]\ge\varepsilon>0
$$
where $H:=h-\mathbf{E}[h|\mathcal{Y}]$.  By approximation
in $L^2$, we may clearly assume without loss of generality that $h$
takes values in $[0,1]$, so that $H$ takes values in $[-1,1]$.
We will fix such a function in the sequel, and consider the loss function
$$
	\ell(u,\omega) = u\,H(\omega)
$$
where we initially choose decisions $u\in[-1,1]$ (the decision space
will be discretized at the end of the proof as required by
Theorem \ref{thm:wkmixcv}).  We claim that for the loss function $\ell$
there exists no weakly pathwise optimal strategy.  This will
be proved by a randomization procedure that will be explained presently.

In the following $([0,1],\mathcal{I})$ denotes the unit interval
with its Borel $\sigma$-field.

\begin{lemma}
\label{lem:couple}
Suppose that $(\Omega,\mathcal{B},\mathbf{P})$ is a standard probability
space.  Then there exists a $(\mathcal{Y}\otimes\mathcal{I})$-measurable 
map $\iota:\Omega\times[0,1]\to\Omega$ such that
$$
	\mathbf{E}[X|\mathcal{Y}](\omega) =
	\int_0^1
	X(\iota(\omega,\lambda))\,d\lambda
	\qquad\mathbf{P}\mbox{-a.e.\ }\omega\in\Omega.
$$
for any bounded ($\mathcal{B}$-)measurable function $X:\Omega\to\mathbb{R}$.
\end{lemma}

\begin{proof}
As $(\Omega,\mathcal{B},\mathbf{P})$ is a standard probability space,
this is \cite[Lemma 3.22]{Kal02} together with the
existence of regular conditional probabilities \cite[Theorem 6.3]{Kal02}.
\qed\end{proof}

Consider the quantity
$$
	A_T^\lambda(\omega) = 
	\frac{1}{T}\sum_{k=1}^T
	H(T^k\iota(\omega,\lambda))\,
	H(T^k\omega).
$$
Then we can compute
\begin{align*}
	\int_0^1(A_T^\lambda)^2\,d\lambda &=
	\frac{1}{T^2}\sum_{m,n=1}^T
	H(T^m\omega)\,
	H(T^n\omega)\int_0^1
	H(T^m\iota(\omega,\lambda))\,
	H(T^n\iota(\omega,\lambda))\,d\lambda \\
	&=
	\frac{1}{T^2}\sum_{m,n=1}^T
	H(T^m\omega)\,
	H(T^n\omega)\,
	\mathbf{E}[\{H\circ T^m\}\{H\circ T^n\}|\mathcal{Y}](\omega).
\end{align*}
In particular, using the invariance of $\mathcal{Y}$, we have
\begin{align*}
	\bigg[\int_0^1\mathbf{E}[(A_T^\lambda)^2]\,d\lambda\bigg]^{1/2} &=
	\mathbf{E}\Bigg[
	\frac{1}{T^2}\sum_{m,n=1}^T
	\mathbf{E}[\{H\circ T^m\}\{H\circ T^n\}|\mathcal{Y}]^2
	\Bigg]^{1/2} 
\displaybreak[0]
\\ &\ge 
	\mathbf{E}\Bigg[
	\frac{1}{T^2}\sum_{m,n=1}^T
	|\mathbf{E}[\{H\circ T^m\}\{H\circ T^n\}|\mathcal{Y}]|
	\Bigg]
\displaybreak[0]
\\ &\ge 
	\mathbf{E}\Bigg[
	\frac{1}{T^2}\sum_{n=1}^T\sum_{m=1}^n
	|\mathbf{E}[\{H\circ T^{m-n}\}\,H|\mathcal{Y}]|
	\Bigg]
\displaybreak[0]
\\ &=
	\mathbf{E}\Bigg[
	\frac{1}{T^2}\sum_{k=0}^{T-1} 
	(T-k)
	|\mathbf{E}[\{H\circ T^{-k}\}\,H|\mathcal{Y}]|
	\Bigg]
\\ &\ge
	\mathbf{E}\Bigg[
	\frac{1}{2T}\sum_{k=0}^{\lfloor T/2\rfloor} 
	|\mathbf{E}[\{H\circ T^{-k}\}\,H|\mathcal{Y}]|
	\Bigg].
\end{align*}
By our choice of $H$, it follows that
$$
	\limsup_{T\to\infty} \mathbf{E}[(A_T^\lambda)^2]\ge
	\frac{\varepsilon^2}{16}
$$
for some $\lambda=\lambda_0\in[0,1]$.  Define
$$
	u_k(\omega) = H(T^k\iota(\omega,\lambda_0)).
$$
Then $u_k$ is $\mathcal{Y}$-measurable for all $k$ (and is therefore 
admissible if we choose, for the time being, the continuous decision 
space $U=[-1,1]$), and $L_T(\mathbf{u}) = A_T^{\lambda_0}$.  Moreover,
$$
	\frac{\varepsilon^2}{16} \le
	\limsup_{T\to\infty}
	\mathbf{E}[(L_T(\mathbf{u}))^2]
	\le \frac{\varepsilon^2}{64} +
	\limsup_{T\to\infty}
	\mathbf{P}\bigg[L_T(\mathbf{u})>\frac{\varepsilon}{8}\bigg] +
	\limsup_{T\to\infty}
	\mathbf{P}\bigg[L_T(\mathbf{u})<-\frac{\varepsilon}{8}\bigg]
$$
implies that we may assume without loss of generality that
$$
	\limsup_{T\to\infty}
	\mathbf{P}\bigg[L_T(\mathbf{u})<-\frac{\varepsilon}{8}\bigg]>0	
$$
(if this is not the case, simply substitute $-\mathbf{u}$ for $\mathbf{u}$
in the following).  But note that the strategy $\mathbf{\tilde u}$ defined
by $\tilde u_k=0$ for all $k$ is mean-optimal (indeed,
$\mathbf{E}[\ell_k(u)|\mathcal{Y}]=u\,\mathbf{E}[H|\mathcal{Y}]\circ T^k
=0$ for all $u$ by construction).  Thus evidently
$$
	\limsup_{T\to\infty}
	\mathbf{P}\bigg[L_T(\mathbf{u})-L_T(\mathbf{\tilde u})
	<-\frac{\varepsilon}{8}\bigg]>0,
$$
so $\mathbf{\tilde u}$ is not weakly pathwise optimal.
It follows from Lemma \ref{lem:wkimplmn} that no weakly pathwise optimal
strategy can exist if we choose the decision space $U=[-1,1]$.

To complete the proof of Theorem \ref{thm:wkmixcv}, it remains to show 
that this conclusion remains valid if we replace $U=[-1,1]$ by some
finite set.
This is easily attained by discretization, however.  Indeed, let
$U=\{k\varepsilon/16:k=-\lfloor 16/\varepsilon\rfloor,\ldots,
\lfloor 16/\varepsilon\rfloor\}$, and construct a new strategy $\mathbf{u}'$
such that $u_k'$ equals the value of $u_k$ (which takes values in $[-1,1]$)
rounded to the nearest element of $U$.  Clearly $\mathbf{\tilde u}$
and $\mathbf{u}'$ both take values in the finite set $U$, and we have
$|L_T(\mathbf{u})-L_T(\mathbf{u}')|\le \varepsilon/16$.  Therefore,
$$
	\limsup_{T\to\infty}
	\mathbf{P}\bigg[L_T(\mathbf{u}')-L_T(\mathbf{\tilde u})
	<-\frac{\varepsilon}{16}\bigg]>0,
$$
and it follows again by Lemma \ref{lem:wkimplmn} that no weakly pathwise optimal
strategy exists.

\subsection{Proof of Theorem \ref{thm:soptinfdim}}
\label{sec:proofinfdim}

By stationarity, we can rewrite the conditional absolute regularity property as
$$
	\big\|
	\mathbf{P}[(X_k)_{k\ge 0}\in\cdot\,|
	\mathcal{X}_{-\infty,-n}\vee\mathcal{Y}_{-\infty,\infty}]
	- \mathbf{P}[(X_k)_{k\ge 0}\in\cdot\,|
	\mathcal{Y}_{-\infty,\infty}]
	\big\|_{\rm TV}\xrightarrow{n\to\infty}0\quad
	\mbox{in }L^1.
$$
Using a simple truncation argument (as the loss is dominated in $L^1$),
this implies
$$
	\esssup_{u\in\mathbb{U}_0}
	\big|\mathbf{E}[l(u,X_0)|\mathcal{X}_{-\infty,-n}\vee\mathcal{Y}_{-\infty,\infty}]
	-\mathbf{E}[l(u,X_0)|\mathcal{Y}_{-\infty,\infty}]\big|
	\xrightarrow{n\to\infty}0\quad\mbox{in }L^1.
$$
If only we could replace $\mathcal{Y}_{-\infty,\infty}$ by
$\mathcal{Y}_{-\infty,0}$ in this expression, all the assumptions
of Theorem \ref{thm:sopt} would follow immediately.  Unfortunately,
it is not immediately obvious whether this replacement is possible
without additional assumptions.

\begin{remark}
In general, it is not clear whether a conditional $K$-automorphism
relative to $\mathcal{Y}_{-\infty,\infty}$ is necessarily a conditional
$K$-automorphism relative to $\mathcal{Y}_{-\infty,0}$.  In this context,
it is interesting to note that the corresponding property does hold for 
conditional weak mixing.  We briefly sketch the proof.  Suppose that
$(\Omega,\mathcal{B},\mathbf{P},T)$ is conditionally weakly mixing
relative to $\mathcal{Y}_{-\infty,\infty}$.  We claim that then also
$$
	\frac{1}{T}\sum_{k=1}^T
	|\mathbf{E}[\{f\circ T^{-k}\}~g|\mathcal{Y}_{-\infty,0}]
	- \mathbf{E}[f\circ T^{-k}|\mathcal{Y}_{-\infty,0}]\,
	\mathbf{E}[g|\mathcal{Y}_{-\infty,0}]
	|
	\xrightarrow{T\to\infty}0
	\quad\mbox{in }L^1
$$
for every $f,g\in L^2$.  Indeed, the conclusion is clearly true whenever
$f$ is $\mathcal{Y}_{-\infty,n}$-measurable for some $n\in\mathbb{Z}$.
By approximation in $L^2$, the conclusion holds whenever $f$ is
$\mathcal{Y}_{-\infty,\infty}$-measurable, and it therefore suffices
to consider $f\in L^2(\mathcal{Y}_{-\infty,\infty})^\perp$.  But in this
case we have $\mathbf{E}[f\circ T^{-k}|\mathcal{Y}_{-\infty,\infty}]=
\mathbf{E}[f\circ T^{-k}|\mathcal{Y}_{-\infty,0}]=0$ for all $k$, and
$$
	\bigg\|
	\frac{1}{T}\sum_{k=1}^T
	|\mathbf{E}[\{f\circ T^{-k}\}~g|\mathcal{Y}_{-\infty,0}]
	|
	\bigg\|_1\le
	\bigg\|
	\frac{1}{T}\sum_{k=1}^T
	|\mathbf{E}[\{f\circ T^{-k}\}~g|\mathcal{Y}_{-\infty,\infty}]
	|
	\bigg\|_1
	\xrightarrow{T\to\infty}0
	\quad\mbox{in }L^1
$$
by Jensen's inequality and
the conditional weak mixing property relative to $\mathcal{Y}_{-\infty,\infty}$.
\end{remark}

As we cannot directly replace $\mathcal{Y}_{-\infty,\infty}$ by
$\mathcal{Y}_{-\infty,0}$, we take an alternative approach.
We begin by noting that, using the conditional absolute regularity
property as described above, we obtain the following trivial adaptation
of Corollary \ref{cor:mtgdiff}.

\begin{lemma}
\label{lem:idimtgdiff}
Under the assumptions of Theorem \ref{thm:soptinfdim}, we have
$$
	\frac{1}{T}\sum_{k=1}^T\{
	l(u_k,X_k)-\mathbf{E}[l(u_k,X_k)|
	\mathcal{Y}_{-\infty,\infty}]\}\xrightarrow{T\to\infty}0
	\quad\mbox{a.s.}
$$
for any admissible strategy $\mathbf{u}$.
\end{lemma}

We will now proceed to replace $\mathcal{Y}_{-\infty,\infty}$ by
$\mathcal{Y}_{-\infty,k}$ in Lemma \ref{lem:idimtgdiff}.
To this end, we use the additional property
established in \cite[Proposition 3.9]{TvH12}:
$$
	\mathbf{P}[(X_k)_{k\le 0}\in\cdot\,|\mathcal{Y}_{-\infty,0}]
	\sim
	\mathbf{P}[(X_k)_{k\le 0}\in\cdot\,|\mathcal{Y}_{-\infty,\infty}]
	\quad\mbox{a.s.}
$$
Theorem \ref{thm:infdim} implies that the past tail $\sigma$-field
$\bigcap_{n}\mathcal{X}_{-\infty,n}$ is $\mathbf{P}[\,\cdot\,|
\mathcal{Y}_{-\infty,0}]$-trivial a.s.\ (cf.\ \cite{Wei83}).
Thus a standard argument \cite[Theorem III.14.10]{Lin02} yields
$$
	\big\|
	\mathbf{P}[(X_k)_{k\le n}\in\cdot\,|\mathcal{Y}_{-\infty,0}]
	-
	\mathbf{P}[(X_k)_{k\le n}\in\cdot\,|\mathcal{Y}_{-\infty,\infty}]
	\big\|_{\rm TV}\xrightarrow{n\to -\infty}0
	\quad\mbox{in }L^1.
$$
Therefore, by stationarity and a simple truncation argument, we have
$$
	\esssup_{u\in\mathbb{U}_0}\big|
	\mathbf{E}[l(u,X_0)|\mathcal{Y}_{-\infty,n}]
	-
	\mathbf{E}[l(u,X_0)|\mathcal{Y}_{-\infty,\infty}]
	\big|\xrightarrow{n\to\infty}0
	\quad\mbox{in }L^1.
$$
This yields the following consequence.

\begin{corollary}
\label{cor:idimtgdiff}
Under the assumptions of Theorem \ref{thm:soptinfdim}, we have
$$
	\frac{1}{T}\sum_{k=1}^T\{
	l(u_k,X_k)-\mathbf{E}[l(u_k,X_k)|
	\mathcal{Y}_{-\infty,k}]\}\xrightarrow{T\to\infty}0
	\quad\mbox{a.s.}
$$
for any admissible strategy $\mathbf{u}$.  In particular,
$$
	\liminf_{T\to\infty} L_T(\mathbf{u}) \ge
	\mathbf{E}\Big[
	\essinf_{u\in\mathbb{U}_0}
	\mathbf{E}[\ell_0(u)|\mathcal{Y}_{-\infty,0}]
	\Big] = L^\star\quad
	\mbox{a.s.}
$$
\end{corollary}

\begin{proof}[Sketch]
Following almost verbatim the proof of Lemma \ref{lem:mtgdiff},
one can prove
$$
	\frac{1}{T}\sum_{k=1}^T\{
	\mathbf{E}[l(u_k,X_k)|\mathcal{Y}_{-\infty,k}]-
	\mathbf{E}[l(u_k,X_k)|
	\mathcal{Y}_{-\infty,k+r}]\}\xrightarrow{T\to\infty}0
	\quad\mbox{a.s.}
$$
for any $r\in\mathbb{N}$.  On the other hand, we have
\begin{align*}
	&\limsup_{T\to\infty}\bigg|
	\frac{1}{T}\sum_{k=1}^T\{
	\mathbf{E}[l(u_k,X_k)|\mathcal{Y}_{-\infty,k+r}]-
	\mathbf{E}[l(u_k,X_k)|
	\mathcal{Y}_{-\infty,\infty}]\}\bigg| \\
	&\quad\le
	\lim_{T\to\infty}
	\frac{1}{T}\sum_{k=1}^T
	\esssup_{u\in\mathbb{U}_k}
	\big|\mathbf{E}[l(u,X_k)|\mathcal{Y}_{-\infty,k+r}]-
	\mathbf{E}[l(u,X_k)|
	\mathcal{Y}_{-\infty,\infty}]
	\big| \\
	&\quad=
	\mathbf{E}\bigg[
	\esssup_{u\in\mathbb{U}_0}
	\big|\mathbf{E}[l(u,X_0)|\mathcal{Y}_{-\infty,r}]-
	\mathbf{E}[l(u,X_0)|
	\mathcal{Y}_{-\infty,\infty}]
	\big|
	\bigg]
	\quad\mbox{a.s.}
\end{align*}
by the ergodic theorem.  It was shown above that the latter quantity
converges to zero as $r\to\infty$, and the result now follows using
Lemma \ref{lem:idimtgdiff}.
\qed\end{proof}

The remainder of the proof of Theorem \ref{thm:soptinfdim} is identical 
to that of Theorem \ref{thm:sopt} modulo trivial modifications, and is 
therefore omitted.

\section*{Acknowledgment}

This work was partially supported by NSF grant DMS-1005575.

\bibliographystyle{spmpsci}
\bibliography{ref}

\end{document}